\documentclass{amsart}
\usepackage[utf8]{inputenc}
\usepackage[T1]{fontenc}
\usepackage{lmodern}
\usepackage{amsmath}
\usepackage{amssymb}
\usepackage{mathtools}
\usepackage{latexsym}
\usepackage[lite]{amsrefs}
\usepackage{nicefrac}
\usepackage{microtype}
\usepackage{color}
\usepackage{tikz-cd}
\usepackage{MnSymbol}
\usepackage{enumitem} 
\setlist[enumerate,1]{label=\textup{(\roman*)}}
\usepackage[pdftitle={Analytic cyclic homology in positive characteristic},
pdfauthor={Ralf Meyer and Devarshi Mukherjee},
pdfsubject={Mathematics}
]{hyperref}

\BibSpec{book}{%
  +{}  {\PrintPrimary}                {transition}
  +{,} { \textit}                     {title}
  +{.} { }                            {part}
  +{:} { \textit}                     {subtitle}
  +{,} { \PrintEdition}               {edition}
  +{}  { \PrintEditorsB}              {editor}
  +{,} { \PrintTranslatorsC}          {translator}
  +{,} { \PrintContributions}         {contribution}
  +{,} { }                            {series}
  +{,} { \voltext}                    {volume}
  +{,} { }                            {publisher}
  +{,} { }                            {organization}
  +{,} { }                            {address}
  +{,} { \PrintDateB}                 {date}
  +{,} { }                            {status}
  +{}  { \parenthesize}               {language}
  +{}  { \PrintTranslation}           {translation}
  +{;} { \PrintReprint}               {reprint}
  +{.} { }                            {note}
  +{.} {}                             {transition}
  +{} { \PrintDOI}                   {doi}
  +{} { available at \url}            {eprint}
  +{}  {\SentenceSpace \PrintReviews} {review}
}

\renewcommand*{\PrintDOI}[1]{\href{http://dx.doi.org/\detokenize{#1}}{doi: \detokenize{#1}}}

\newcommand{\comment}[1]{}  

\theoremstyle{plain}
\newtheorem{theorem}{Theorem}[section]
\newtheorem{lemma}[theorem]{Lemma}
\newtheorem{corollary}[theorem]{Corollary}
\newtheorem{proposition}[theorem]{Proposition}
\theoremstyle{remark}
\newtheorem{remark}[theorem]{Remark}
\theoremstyle{definition}
\newtheorem{definition}[theorem]{Definition}
\newtheorem{example}[theorem]{Example}
\newtheorem{step}{Step}
\numberwithin{theorem}{section}

\newcommand\N{\mathbb N}
\newcommand\Q{\mathbb Q}

\newcommand\Z{\mathbb Z}

\newcommand{\coma}{\widehat}
\newcommand{\comb}{\overbracket[.7pt][1.4pt]}

\newcommand*{\tub}[2]{\mathcal{U}(#1,#2)}
\newcommand*{\tens}{\mathsf{T}}
\newcommand*{\jens}{\mathsf{J}}

\newcommand*{\cone}{\mathsf{cone}}

\newcommand{\updagger}{\textup{\tiny\!\!\dagger}}

\newcommand{\idealin}{\mathrel{\triangleleft}} 

\newcommand{\diff}{\mathrm{d}}
\newcommand{\Cont}{\mathrm{C}}
\newcommand{\defeq}{\mathrel{:=}} 

\newcommand*{\into}{\rightarrowtail}
\newcommand*{\onto}{\twoheadrightarrow}

\newcommand*{\ling}[1]{#1_\mathrm{lg}}

\DeclarePairedDelimiter{\abs}{\lvert}{\rvert}
\DeclarePairedDelimiter{\norm}{\lVert}{\rVert}
\DeclarePairedDelimiter{\floor}{\lfloor}{\rfloor}
\DeclarePairedDelimiterX{\setgiven}[2]{\{}{\}}{#1\,{:}\,\mathopen{}#2}

\newcommand\hot{\mathbin{\comb{\otimes}}}
\newcommand\haotimes{\mathbin{\coma{\otimes}}}

\DeclareMathOperator{\coker}{coker}

\DeclareMathOperator{\holim}{holim}
\DeclareMathOperator{\qholim}{``holim''}

\DeclareMathOperator{\HA}{HA}
\DeclareMathOperator{\HAC}{\mathbb{HA}}

\newcommand{\fake}{\mathop{\text{``}\!\prod\!\text{''}}}

\newcommand{\ev}{\mathrm{ev}}
\newcommand{\even}{\mathrm{ev}}
\newcommand{\odd}{\mathrm{odd}}
\newcommand{\id}{\mathrm{id}}
\newcommand{\nb}{\nobreakdash}
\newcommand{\dvr}{V}
\newcommand{\dvgen}{\pi}
\newcommand{\dvf}{F}
\newcommand{\resf}{\mathbb F}

\DeclareMathOperator{\Hom}{Hom}
\DeclareMathOperator{\HP}{HP}
\DeclareMathOperator{\diss}{diss}

\begin{document}
\title[Analytic cyclic homology in positive characteristic]
{Analytic cyclic homology\\in positive characteristic}

\author{Ralf Meyer}
\email{rmeyer2@uni-goettingen.de}

\address{Mathematisches Institut\\
  Universit\"at Göttingen\\
  Bunsenstra\ss{}e 3--5\\
  37073 Göttingen\\
  Germany}

\author{Devarshi Mukherjee}
\email{dmukherjee@dm.uba.ar}

\address{Dep. Matemática-IMAS\\
 FCEyN-UBA, Ciudad Universitaria Pab 1\\
 1428 Buenos Aires\\ 
Argentina}

\begin{abstract}
  Let~\(\dvr\) be a complete discrete valuation ring with
  residue field~\(\resf\).  We define a cyclic homology theory
  for algebras over~\(\resf\), by lifting them to free algebras
  over~\(\dvr\), which we enlarge to tube algebras and complete
  suitably.  We show that this theory may be computed using any
  pro-dagger algebra lifting of an \(\resf\)\nb-algebra.  We show
  that our theory is polynomially homotopy invariant, excisive, and
  matricially stable.
\end{abstract}

\thanks{The second named author would like to thank Ryszard Nest and Guillermo Corti\~nas for valuable insights.  This research
  is partially supported by the Center for Symmetry and Deformation,
  Copenhagen, Denmark and the Feodor-Lynen Fellowship of the Alexander von Humboldt Foundation. The authors thank the anonymous referee for helpful comments.}




\maketitle

\tableofcontents

\section{Introduction}

Let~\(\dvr\) be a complete discrete valuation ring.  Let~\(\dvgen\)
denote a uniformiser, \(\resf\) its residue field, and~\(\dvf\)
its fraction field.  \emph{We assume~\(\dvf\) to have
  characteristic~\(0\) throughout this article.}

Cyclic homology and its variants are fundamental invariants in
noncommutative geometry.  They were developed by Connes and Tsygan
in the 1980s and vastly generalise de Rham cohomology for smooth
manifolds and schemes to noncommutative algebras.  This article is a
step in a programme that aims at defining a version of cyclic
homology for \(\resf\)\nb-algebras that specialises to rigid
cohomology in the commutative case.  Here we propose such a
definition and prove that our theory has good homological
properties and gives rigid cohomology at least for
smooth curves over~\(\resf\).  As a highly noncommutative example,
we compute our theory for Leavitt and Cohn path algebras
over~\(\resf\).  Other important examples that deserve to be treated
in the future are the crossed product algebras attached to orbifolds
(actions of finite groups on affine varieties over~\(\resf\)) and
noncommutative tori.

The first step in the programme mentioned above is achieved
in~\cite{Cortinas-Cuntz-Meyer-Tamme:Nonarchimedean}, by identifying
Berthelot's rigid cohomology for finite type commutative
\(\resf\)\nb-algebras with the periodic cyclic homology of certain
dagger completed liftings to \(\dvr\)\nb-algebras.  The tube algebra
construction and the dagger completion in the realm of bornological
algebras introduced
in~\cite{Cortinas-Cuntz-Meyer-Tamme:Nonarchimedean} are crucial
ingredients in our construction as well.  Another important
technical ingredient that is still missing
in~\cite{Cortinas-Cuntz-Meyer-Tamme:Nonarchimedean} is the concept
of bornological torsionfreeness introduced
in~\cite{Meyer-Mukherjee:Bornological_tf}.  The second step in the
programme is the definition of an analytic cyclic homology theory
for projective systems of dagger algebras
in~\cite{Cortinas-Meyer-Mukherjee:NAHA}.

The periodic cyclic homology approach
in~\cite{Cortinas-Cuntz-Meyer-Tamme:Nonarchimedean} gives a chain
complex computing rigid cohomology that is exactly functorial, not
just up to homotopy equivalence.  This makes it easy to glue the
chain complexes on the charts of a scheme and is an advantage over
other definitions of rigid cohomology that depend on the choice of
a lifting over~\(\dvr\).  Our theory is also defined by attaching a
chain complex \(\HAC(A)\) to an \(\resf\)\nb-algebra~\(A\) in a
natural way, so that it could generalise to an invariant of schemes.

Our
main achievement is to prove that \(\HAC(A)\) may be computed using
any dagger algebra lifting of~\(A\).  Namely, let~\(D\) be a dagger
algebra, that is, a suitably complete bornological algebra
over~\(\dvr\), and let \(A\defeq D/\dvgen D\) be its reduction
mod~\(\dvgen\), which is an \(\resf\)\nb-algebra.  We assume that
the quotient bornology on \(D/\dvgen D\) is the fine one and briefly
call~\(D\) fine mod~\(\dvgen\).  Then the chain complex \(\HAC(D)\)
defined in~\cite{Cortinas-Meyer-Mukherjee:NAHA} is quasi-isomorphic
to the chain complex \(\HAC(A)\) defined here.  More generally, we
prove this in the situation of an extension of projective systems of
bornological algebras \(N \into D \onto A\), where~\(D\) is a
pro-dagger algebra and fine mod~\(\dvgen\) and~\(N\) is analytically
nilpotent (see~\cite{Cortinas-Meyer-Mukherjee:NAHA}).

We also prove that such an extension \(N \into D \onto A\) exists for any
\(\resf\)\nb-algebra~\(A\).  This is the key to proving that our
theory satisfies excision and is matrix-stable and homotopy
invariant.  The computation of analytic cyclic homology for Leavitt
path algebras follows easily from these formal properties combined
with~\cite{Cortinas-Montero:K_Leavitt}.  This example highlights the
power of seemingly formal properties as a computational tool.  This
insight is the motivation for the definition of bivariant
K\nb-theory on categories of rings and algebras (see
\cites{Cuntz:Bivariante, Cortinas-Thom:Bivariant_K}).

If \(D_1\) and~\(D_2\) are two dagger algebras with
\(D_1/\dvgen D_1 \cong D_2/\dvgen D_2\), then it follows from the
theorem mentioned above that \(\HAC(D_1)\) and \(\HAC(D_2)\) are
quasi-isomorphic.  This becomes wrong for periodic cyclic homology.
The difference between periodic and analytic cyclic homology is
subtle and technical.  This is the reason why we cannot prove yet
that our theory generalises rigid cohomology for commutative
\(\resf\)\nb-algebras of finite type.  It is shown
in~\cite{Cortinas-Cuntz-Meyer-Tamme:Nonarchimedean} that rigid
cohomology is isomorphic to \(\HP_*(D)\) for a suitable dagger
algebra lifting~\(D\), while we prove here that our theory agrees
with \(\HA_*(D)\).  We can only identify \(\HP_*(D)\) with
\(\HA_*(D)\) if~\(D\) is smooth of relative dimension~\(1\)
over~\(\dvr\), that is, when we start with a smooth curve
over~\(\resf\).

Starting with the definition
in~\cite{Cortinas-Meyer-Mukherjee:NAHA}, it is very hard to prove
that \(\HAC(D)\) for a pro-dagger algebra that is fine
mod~\(\dvgen\) depends only on \(D/\dvgen D\).  Our proof of this
depends on locally defined maps, that is, families of partial maps
defined only on bounded \(\dvr\)\nb-submodules, without any
compatibility between these partial maps.  We use a different
completion, taken in the realm of inductive systems of Banach
\(\dvf\)\nb-vector spaces, in order to handle such maps.  In the
end, our main theorem implies that completing in the realm of
bornological \(\dvf\)\nb-vector spaces gives a quasi-isomorphic
chain complex.  We cannot prove this, however, without using
inductive systems in a key technical step.

The paper is organised as follows.

In Section~\ref{sec:preliminaries}, we review some basics on
bornologies and their relationship with inductive
systems.  We compare the completions of
bornological \(\dvf\)\nb-vector spaces and inductive systems of
seminormed \(\dvf\)\nb-vector spaces and prove a criterion when the
dissection functors from bornological vector spaces to inductive
systems intertwine these two different completion functors.

In Section~\ref{sec:definition_HA}, we define the analytic cyclic
complex \(\HAC(A)\) of an \(\resf\)\nb-algebra~\(A\), its homology
\(\HA_*(A)\), and the bivariant analytic cyclic homology
\(\HA_*(A,B)\) for two \(\resf\)\nb-algebras \(A\) and~\(B\).  To
define \(\HA_*(A,B)\), we observe that the category of projective
systems of inductive systems of Banach \(\dvf\)\nb-vector spaces is
quasi-Abelian, so that its derived category is defined.  The
construction of \(\HAC(A)\) has many small steps as
in~\cite{Cortinas-Meyer-Mukherjee:NAHA}.  Compared
to~\cite{Cortinas-Meyer-Mukherjee:NAHA}, we use
a different completion, as explained above, and
some steps occur in a different order in order to stay longer with
bornologies.

Section~\ref{sec:compare_HA_theories} compares our new theory to the
analytic cyclic homology theory for pro-dagger algebras defined
in~\cite{Cortinas-Meyer-Mukherjee:NAHA}.  This section is where we
state the theorems about quasi-isomorphisms
\(\HAC(A) \simeq \HAC(D)\).  It only contains the easier parts of
the proofs.  The key steps in the proofs are postponed to
Section~\ref{sec:independence_lifting}.  This is where we build
partial homomorphisms between different \(\dvr\)\nb-algebras related
to~\(A\).  The more technical statements in
Section~\ref{sec:compare_HA_theories} are based on a variant of
\(\HAC(A)\) that depends on a projective system of complete,
torsionfree bornological \(\dvr\)\nb-modules~\(W\) with a morphism
\(\varrho\colon W \to A\).  This theory is defined in
Section~\ref{sec:homotopy_stability}, by changing only the very
first step in the construction of \(\HAC(A)\).  In the easiest case,
\(W\) could be a dagger algebra and~\(\varrho\) could come from an
isomorphism \(D/\dvgen D\cong A\).

Section~\ref{sec:lift_pro_dagger} proves that any
\(\resf\)\nb-algebra~\(A\) is part of an extension
\(N \into D \onto A\), where \(D=(D_n)_{n\in\N}\) is a projective
system of dagger algebras and fine mod~\(\dvgen\), and~\(N\) is
analytically nilpotent.  In
this situation, \(\HAC(A)\) is quasi-isomorphic to \(\HAC(D)\).  The
proof is subtle because the universal property of the tube algebra
only allows us to build maps to torsionfree algebras, not to~\(A\).
In Section~\ref{sec:homological_properties}, we prove that
(bivariant) analytic cyclic homology for \(\resf\)\nb-algebras is
homotopy invariant, excisive and matricially stable.  We use the
quasi-isomorphisms \(\HAC(A) \simeq\HAC(D)\) proven in
Section~\ref{sec:lift_pro_dagger} to carry over the analogous
results for projective systems of dagger algebras proven
in~\cite{Cortinas-Meyer-Mukherjee:NAHA}.

\section{Bornologies and inductive systems}
\label{sec:preliminaries}

We recall the category of bornological \(\dvr\)\nb-modules and some
important subcategories, and compare these to related categories of
inductive systems of \(\dvr\)\nb-modules.  We also compare the
completions in the categories of bornological \(\dvf\)\nb-vector
spaces and inductive systems of seminormed \(\dvf\)\nb-vector
spaces.  This is needed to compare our new definition with the one
used in \cites{Cortinas-Cuntz-Meyer-Tamme:Nonarchimedean,
  Cortinas-Meyer-Mukherjee:NAHA}.  In later sections, this will be
used to carry over theorems about excision and matrix stability from
the analytic cyclic homology for pro-dagger algebras defined
in~\cite{Cortinas-Meyer-Mukherjee:NAHA} to the theory defined here.
The results in~\cite{Cortinas-Cuntz-Meyer-Tamme:Nonarchimedean} that
compare bornological \(\dvr\)\nb-modules to inductive systems of
\(\dvr\)\nb-modules are transferred to \(\dvf\)\nb-vector spaces
because this version will be needed later.

A \emph{bornology} on a set is a collection of subsets, called
\emph{bounded subsets}, such that finite unions and subsets of
bounded subsets are bounded, and finite sets are bounded.  A map
between bornological sets is \emph{bounded} if it maps bounded
subsets to bounded subsets.  A \emph{bornological
  \(\dvr\)\nb-module} is a \(\dvr\)\nb-module~\(M\) with a bornology
with the extra property that the \(\dvr\)\nb-submodule of~\(M\)
generated by a bounded subset is again bounded.  Let
\(\mathsf{Bor}_\dvr\) be the category of bornological
\(\dvr\)\nb-modules and bounded \(\dvr\)\nb-module homomorphisms.

A \emph{bornological \(\dvf\)\nb-vector space} is the same as a
bornological \(\dvr\)\nb-module with the extra property that
multiplication by~\(\dvgen\) is a bornological isomorphism.  In this
way, we identify the category \(\mathsf{Bor}_\dvf\) of bornological
\(\dvf\)\nb-vector spaces with a full subcategory of
\(\mathsf{Bor}_\dvr\).  A bornological \(\dvr\)\nb-module is
\emph{bornologically torsionfree} if multiplication by~\(\dvgen\) is
a bornological embedding
(see~\cite{Meyer-Mukherjee:Bornological_tf}).  \emph{As
  in~\cite{Cortinas-Meyer-Mukherjee:NAHA}, we lighten our notation
  by abbreviating ``bornologically torsionfree'' to ``torsionfree''
  for bornological \(\dvr\)\nb-modules.}

A bornological \(\dvr\)\nb-module~\(M\) is \emph{complete} if every
bounded submodule of~\(M\) is contained in a \(\dvgen\)\nb-adically
complete, bounded \(\dvr\)\nb-submodule.  A bornological
\(\dvr\)\nb-module~\(M\) is \emph{separated} if every bounded
submodule of~\(M\) is Hausdorff in the
\(\dvgen\)\nb-adic topology.  Complete bornological
\(\dvr\)\nb-modules are separated.  Let
\(\mathsf{CBor}_\dvf \subseteq \mathsf{Bor}_\dvf\) be the full
subcategory of complete bornological \(\dvf\)\nb-vector spaces.
The \emph{completion}
\[
  \comb{\phantom{I}}\colon \mathsf{Bor}_\dvf \to
  \mathsf{CBor}_\dvf,\qquad
  M\mapsto \comb{M},
\]
is defined as the left adjoint functor of the inclusion
\(\mathsf{CBor}_\dvf \hookrightarrow \mathsf{Bor}_\dvf\).  This
completion is the restriction of the completion for bornological
\(\dvr\)\nb-modules (see
\cite{Cortinas-Cuntz-Meyer-Tamme:Nonarchimedean}*{Definition~2.14});
this follows because the property of being a bornological
\(\dvf\)\nb-vector space is preserved by the completion
in~\cite{Cortinas-Cuntz-Meyer-Tamme:Nonarchimedean}.

We often use the \emph{fine bornology} on a
\(\dvr\)\nb-module~\(M\).  It is the smallest bornology: a subset is
bounded if and only if it is contained in a finitely generated
\(\dvr\)\nb-submodule.  This bornology is always complete, and it
makes~\(M\) bornologically torsionfree if~\(M\) is torsionfree as a
\(\dvr\)\nb-module.

The category of bornological \(\dvr\)\nb-modules is identified in
\cite{Cortinas-Cuntz-Meyer-Tamme:Nonarchimedean}*{Proposition~2.5}
with a full subcategory of the category of inductive systems of
\(\dvr\)\nb-modules.  We now carry this over to bornological
\(\dvf\)\nb-vector spaces.  Let \(\mathsf{Norm}_\dvf^{1/2}\) be the
category of seminormed \(\dvf\)\nb-vector spaces with bounded
\(\dvf\)\nb-linear maps as arrows.  Let
\(\mathsf{Ind}(\mathsf{Norm}_\dvf^{1/2})\) be the category of
inductive systems over it.  This is the appropriate target category
for the dissection functor on \(\mathsf{Bor}_\dvf\) because we want
the \(\dvf\)\nb-vector space structure to be a levelwise feature of
the inductive systems.

Let~\(M\) be a bornological \(\dvf\)\nb-vector space and let
\(S\subseteq M\) be a bounded \(\dvr\)\nb-submodule.  Then
\(S\otimes \dvf \subseteq M\), and we give \(S\otimes \dvf\) the
\emph{gauge seminorm} of~\(S\); this is the unique seminorm that
has~\(S\) as its unit ball.  The bounded \(\dvr\)\nb-submodules form
a directed set, and \(S\mapsto S\otimes \dvf\) is an inductive
system of seminormed \(\dvf\)\nb-vector spaces indexed by it.  The
\emph{dissection functor}
\begin{equation}
  \label{eq:dissection_over_F}
  \diss\colon \mathsf{Bor}_\dvf \to
  \mathsf{Ind}(\mathsf{Norm}_\dvf^{1/2})
\end{equation}
maps a bornological \(\dvf\)\nb-vector space~\(M\) to this inductive
system.  The inductive system \(\diss(M)\) has the extra property
that the structure maps \(S\otimes \dvf \to T\otimes \dvf\) for
\(S\subseteq T\) are injective and contractive.

The category of bornological \(\dvf\)\nb-vector spaces is complete
and cocomplete.  Taking inductive limits defines a functor
\begin{equation}
  \label{eq:injlim_over_F}
  \varinjlim\colon
  \mathsf{Ind}(\mathsf{Norm}_\dvf^{1/2}) \to  \mathsf{Bor}_\dvf.
\end{equation}

\begin{lemma}
  \label{lem:diss_injlim_1}
  Let~\(M\) be a bornological \(\dvf\)\nb-vector space.  There is a
  natural isomorphism \(\varinjlim \diss(M) \cong M\).
\end{lemma}

\begin{proof}
  The inclusions \(S\otimes \dvf \to M\) for bounded
  \(\dvr\)\nb-submodules \(S\subseteq M\) combine to a canonical
  injective map \(\varinjlim \diss(M) \to M\).  Since any bounded
  subset of~\(M\) is contained in a bounded \(\dvr\)\nb-submodule,
  this map is a bornological isomorphism.
\end{proof}

\begin{lemma}
  \label{lem:diss_injlim_2}
  Let \(M = (M_i)_{i\in I}\) be an inductive system of seminormed
  \(\dvf\)\nb-vector spaces.  There is a canonical map
  \(\vartheta\colon M \to \diss \bigl(\varinjlim M)\).  The following are
  equivalent:
  \begin{enumerate}
  \item \label{lem:diss_injlim_2_1}%
    \(\vartheta\) is an isomorphism;
  \item  \label{lem:diss_injlim_2_2}%
    \(M\) is isomorphic to an inductive system with injective
    structure maps;
  \item \label{lem:diss_injlim_2_3}%
    for each \(i\in I\), there is \(j\in I_{\ge i}\) such that for
    all \(k\in I_{\ge j}\), the structure maps \(M_i \to M_j\) and
    \(M_i \to M_k\) have the same kernel.
  \end{enumerate}
\end{lemma}

\begin{proof}
  Let \(B(M_i)\) denote the unit ball of~\(M_i\).  Let
  \(\varphi_i\colon M_i \to \varinjlim M\) for \(i\in I\) be the
  canonical maps.  Any bounded \(\dvr\)\nb-submodule
  of~\(\varinjlim M\) is contained in
  \(\varphi_i(\dvgen^{-n}\cdot B(M_i))\) for some \(n\in\N\),
  \(i\in I\).  Therefore, \(\diss \varinjlim M\) is
  isomorphic to the inductive system of seminormed spaces
  \(\varphi_i(B(M_i))\otimes \dvf = \varphi_i(M_i)\), equipped with
  the gauge seminorm of \(\varphi_i(\dvgen^{-n}B(M_i))\), indexed by
  \(\N\times I\).  Actually, since rescaling the norm gives an
  isomorphic seminormed vector space, we may drop the factor~\(\N\)
  and reindex \(\diss \varinjlim M\) using~\(I\) itself.
  The canonical maps \(M_i \to \varphi_i(M_i)\) combine to a
  canonical morphism
  \(\vartheta\colon M \to \diss \varinjlim M\) in
  \(\mathsf{Ind}(\mathsf{Norm}_\dvf^{1/2})\).

  If~\(M\) has injective structure maps, then the maps
  \(M_i \to \varphi_i(M_i)\) above are isometric isomorphisms for
  all \(i\in I\) (and \(n=0\)).  Then~\(\vartheta\) is an
  isomorphism in \(\mathsf{Ind}(\mathsf{Norm}_\dvf^{1/2})\).
  Since~\(\vartheta\) is natural, it remains an isomorphism if~\(M\)
  is only \emph{isomorphic} to an inductive system with injective
  structure
  maps.  Since the structure maps in~\(\diss (W)\) for a
  bornological \(\dvf\)\nb-vector space are injective,
  \ref{lem:diss_injlim_2_1} and~\ref{lem:diss_injlim_2_2} are
  equivalent.

  Next we prove that \ref{lem:diss_injlim_2_2} implies
  \ref{lem:diss_injlim_2_3}.  Let \(N= (N_\alpha)_{\alpha\in J}\) be
  an inductive system of seminormed \(\dvf\)\nb-vector spaces with
  injective structure maps that is isomorphic to~\(M\).  For
  \(i\in I\) and \(\alpha\in J\), the isomorphism \(M\cong N\) gives
  us maps \(M_i \to N_{\alpha(i)}\) and
  \(N_\alpha \to M_{j(\alpha)}\) for some \(\alpha(i)\in J\),
  \(j(\alpha)\in I\).  There is \(j\ge j(\alpha(i))\) such that the
  composite map
  \(M_i \to N_{\alpha(i)} \to M_{j(\alpha(i))} \to M_j\) is the
  structure map of the inductive system.  Let \(k\ge j\).  Then
  there is \(\alpha \ge \alpha(k)\) so that the composite maps
  \(M_i \to M_k \to N_{\alpha(k)} \hookrightarrow N_\alpha\) and
  \(M_i \to N_{\alpha(i)} \hookrightarrow N_\alpha\) are equal.
  Since the map \(N_{\alpha(i)} \hookrightarrow N_\alpha\) is
  injective, the kernel of the map \(M_i \to M_k\) is contained in
  the kernel of the map \(M_i \to N_{\alpha(i)}\), which is
  contained in the kernel of the structure map \(M_i \to M_j\).

  Finally, we prove that~\ref{lem:diss_injlim_2_3}
  implies~\ref{lem:diss_injlim_2_2}.  For \(i\in I\), let
  \(K_i \defeq \bigcup_{j\in I_{\ge i}} \ker (M_i \to M_j)\).  Let
  \(M'_i \defeq M_i / K_i\).  The structure maps \(M_i \to M_j\) for
  \(i \le j\) descend to injective maps
  \(M'_i \hookrightarrow M'_j\).  These form an inductive
  system~\(M'\) of seminormed \(\dvf\)\nb-vector spaces.  The
  quotient maps \(M_i \to M'_i\) form a morphism of inductive
  systems.  The condition~\ref{lem:diss_injlim_2_3} implies that
  \(K_i = \ker (M_i \to M_j)\) for some \(j\in I_{\ge i}\).
  Therefore, the structure map \(M_i \to M_j\) descends to a bounded
  \(\dvf\)\nb-linear map \(M'_i \to M_j\).  These maps form a
  morphism of inductive systems, which is inverse to the canonical
  morphism \(M\to M'\).
\end{proof}

\begin{proposition}
  \label{pro:diss_injlim_Bor}
  The inductive limit functor in~\eqref{eq:injlim_over_F} is left
  adjoint to the dissection functor in~\eqref{eq:dissection_over_F}.
  The dissection functor restricts to an equivalence of categories
  between \(\mathsf{Bor}_\dvf\) and the full subcategory of
  \(\mathsf{Ind}(\mathsf{Norm}_\dvf^{1/2})\) consisting of those
  inductive systems with injective structure maps.  Alternatively,
  we may use the full subcategory of those inductive systems that
  satisfy the equivalent conditions in
  Lemma~\textup{\ref{lem:diss_injlim_2}}.
\end{proposition}

\begin{proof}
  The canonical map \(\vartheta\colon M \to \diss \varinjlim M\) for
  \(M\in \mathsf{Ind}(\mathsf{Norm}_\dvf^{1/2})\) and the canonical
  isomorphism \(\varinjlim \diss(M) \cong M\) for
  \(M\in \mathsf{Bor}_\dvf\) in Lemmas \ref{lem:diss_injlim_1}
  and~\ref{lem:diss_injlim_2} satisfy the triangle identities for
  the unit and counit of an adjunction (see
  \cite{Riehl:Categories_context}*{Definition 4.2.5}).  Since the
  counit is a natural isomorphism, it follows that \(\diss\) is an
  equivalence of categories onto its essential range.  The essential
  range is described in Lemma~\ref{lem:diss_injlim_2}, which also
  shows that it is equivalent to the full subcategory of inductive
  systems with injective structure maps.
\end{proof}

Next we adapt Proposition~\ref{pro:diss_injlim_Bor} to the category
\(\mathsf{CBor}_\dvf\) of complete bornological \(\dvf\)\nb-vector
spaces.  Let \(\mathsf{Ban}_\dvf\) be the category of Banach
\(\dvf\)\nb-vector spaces with bounded \(\dvf\)\nb-linear maps as
arrows, and let \(\mathsf{Ind}(\mathsf{Ban}_\dvf)\) be its category
of inductive systems.  The dissection functor
\begin{equation}
  \label{eq:dissection_over_F_complete}
  \diss\colon \mathsf{CBor}_\dvf \to
  \mathsf{Ind}(\mathsf{Ban}_\dvf)
\end{equation}
is defined similarly, but we only use \(\dvgen\)\nb-adically
complete bounded \(\dvr\)\nb-submodules \(S\subseteq M\) instead.
If~\(M\) is complete, then any bounded subset of~\(M\) is contained
in a \(\dvgen\)\nb-adically complete bounded \(\dvr\)\nb-submodule.
Therefore, the directed set used in the ``complete'' dissection
functor~\eqref{eq:dissection_over_F_complete} is cofinal in the
directed set used in the ``incomplete'' dissection
functor~\eqref{eq:dissection_over_F}.  Therefore, both definitions
give naturally isomorphic inductive systems.  We will not
distinguish them in our notation.
The dissection functor in~\eqref{eq:dissection_over_F_complete} is
right adjoint to the separated inductive limit functor
\begin{equation}
  \label{eq:injlim_over_F_complete}
  \varinjlim\nolimits^s \colon
  \mathsf{Ind}(\mathsf{Ban}_\dvf) \to  \mathsf{CBor}_\dvf.
\end{equation}
This maps an inductive system to the \emph{separated quotient} of
its inductive limit in \(\mathsf{Bor}_\dvf\); the separated quotient
is the quotient by~\(\overline{\{0\}}\), the bornological closure
of~\(0\).  This is the largest quotient of the usual inductive limit
that is complete in its natural bornology.  The separated inductive
limit is the inductive limit in the categorical sense in
\(\mathsf{CBor}_\dvf\).  If~\(M\) is an inductive system of Banach
spaces with injective structure maps, then \(\varinjlim M\) is
already separated, so that \(\varinjlim\nolimits^s M = \varinjlim M\).

\begin{proposition}
  \label{pro:diss_injlim_CBor}
  The inductive limit functor in~\eqref{eq:injlim_over_F_complete}
  is left adjoint to the dissection functor
  in~\eqref{eq:dissection_over_F_complete}.  The dissection functor
  restricts to an equivalence of categories between
  \(\mathsf{CBor}_\dvf\) and the full subcategory of
  \(\mathsf{Ind}(\mathsf{Ban}_\dvf)\) consisting of those inductive
  systems with injective structure maps.  Alternatively, we may use
  the full subcategory of those inductive systems of Banach
  \(\dvf\)\nb-vector spaces that satisfy the equivalent conditions
  in Lemma~\textup{\ref{lem:diss_injlim_2}}.
\end{proposition}

\begin{proof}
  The proof of Proposition~\ref{pro:diss_injlim_Bor} carries over.
\end{proof}

The \emph{completion} of an inductive system
\((M_i ,\norm{\cdot}_i)_{i \in I}\) of seminormed \(\dvf\)\nb-vector
spaces is the inductive system of Banach \(\dvf\)\nb-vector spaces
\(\bigl(\comb{M_i}\bigr)_{i\in I}\), where~\(\comb{M_i}\) is the
Hausdorff completion of~\(M_i\).  This defines a functor
\(\mathsf{Ind}(\mathsf{Norm}_\dvf^{1/2}) \to
\mathsf{Ind}(\mathsf{Ban}_\dvf)\) that is left adjoint to the
inclusion functor
\(\mathsf{Ind}(\mathsf{Ban}_\dvf) \hookrightarrow
\mathsf{Ind}(\mathsf{Norm}_\dvf^{1/2})\).

The completions in \(\mathsf{Bor}_\dvf\) and
\(\mathsf{Ind}(\mathsf{Norm}_\dvf^{1/2})\) are related by a natural
isomorphism
\begin{equation}
  \label{eq:compare_completions}
  \comb{M}
  \cong \varinjlim\nolimits^s \comb{\diss(M)}
  = \varinjlim \comb{\diss(M)} \bigm/ \overline{\{0\}}.
\end{equation}
This follows because
\(M\mapsto \varinjlim\nolimits^s \comb{\diss(M)}\) is a
left adjoint functor to the inclusion
\(\mathsf{CBor}_\dvf \hookrightarrow \mathsf{Bor}_\dvf\).
The unit of the adjunction
\begin{equation}
  \label{eq:diss_lim_adjunction_CBor}
  \begin{tikzcd}
    \mathsf{CBor}_\dvf
    \ar[r, shift right=1ex, "\diss"']
    \ar[r, phantom,  "\bot"]
    &
    \mathsf{Ind}(\mathsf{Ban}_\dvf)
    \ar[l, shift right=1ex, swap,  "\varinjlim\nolimits^s"]
  \end{tikzcd}
\end{equation}
gives a natural map
\begin{equation}
  \label{eq:diss_comb_to_comb_diss}
  \comb{\diss(M)} \to
  \diss \bigl(\varinjlim\nolimits^s \comb{\diss(M)}\bigr)
  = \diss \comb{M}.
\end{equation}

\begin{proposition}
  \label{pro:diss_comb}
  Let \(M\in \mathsf{Bor}_\dvf\).  The following are equivalent:
  \begin{enumerate}
  \item \label{en:diss_comb_1}%
    the canonical map \(\comb{\diss(M)} \to \diss \comb{M}\)
    in~\eqref{eq:diss_comb_to_comb_diss} is an isomorphism;
  \item \label{en:diss_comb_2}%
    \(\comb{\diss(M)}\) is in the essential range of
    \(\diss\colon \mathsf{CBor}_\dvf \to
    \mathsf{Ind}(\mathsf{Ban}_\dvf)\);
  \item \label{en:diss_comb_3}%
    for any bounded \(\dvr\)\nb-submodule \(S\subseteq M\), there is
    a bounded \(\dvr\)\nb-submodule \(T\subseteq M\) with
    \(S\subseteq T\) such that for any bounded \(\dvr\)\nb-submodule
    \(U\subseteq M\) with \(T\subseteq U\), the maps
    \(\coma{S} \to \coma{T}\) and \(\coma{S} \to \coma{U}\) have the
    same kernel; here \(\coma{S}\) is the \(\dvgen\)\nb-adic
    completion of~\(S\).
  \end{enumerate}
\end{proposition}

\begin{proof}
  This follows from the equivalence between the three conditions in
  Lemma~\ref{lem:diss_injlim_2}, applied to inductive systems of
  Banach spaces.  We only have to check that the third condition in
  this proposition translates to the third condition in
  Lemma~\ref{lem:diss_injlim_2}.  Indeed, since~\(\coma{S}\) is
  isomorphic to the unit ball of \(\comb{S\otimes\dvf}\), the maps
  \(\coma{S} \to \coma{T}\) and \(\coma{S} \to \coma{U}\) have the
  same kernel if and only if
  \(\ker \bigl(\comb{S\otimes\dvf} \to \comb{T\otimes\dvf}\bigr) =
  \ker \bigl(\comb{S\otimes\dvf} \to \comb{U\otimes\dvf}\bigr)\).
\end{proof}

The advantage of working in the category of inductive systems of
Banach \(\dvf\)\nb-vector spaces is that its completion functor is
just levelwise.  This variant of the completion will be used in the
next section, and it deserves a name of its own:

\begin{definition}
  \label{quasi:completion}
  The \emph{quasi-completion} of a bornological \(\dvf\)\nb-vector
  space~\(M\) is the inductive system of Banach \(\dvf\)\nb-vector
  spaces \(\comb{\diss(M)}\).
\end{definition}

We will need the following in the proof of Theorem \ref{the:homotopy_invariance}.

\begin{lemma}
  \label{lem:tensor-exact}
  Let \(f \colon M \to N\) be a bornological embedding between
  complete, bornologically torsionfree \(\dvr\)-modules.  Then for
  an arbitrary complete, torsionfree bornological
  \(\dvr\)\nb-module~\(D\), the induced map
  \(f \otimes 1_D \colon M \hot D \to N \hot D\) is a bornological
  embedding.
\end{lemma}

\begin{proof}
  Since~\(D\) is complete and bornologically torsionfree, we may
  write it as an inductive limit
  \(D \cong \varinjlim_i \Cont_0(X_i,\dvr)\) of unit balls of Banach
  \(\dvf\)\nb-vector spaces, with bounded, injective structure maps.
  This is \cite{Cortinas-Meyer-Mukherjee:NAHA}*{Corollary~2.4.3}.
  For each~\(i\), the embedding~\(f\) induces a bornological
  embedding from
  \(\Cont_0(X_i,\dvr) \haotimes M \cong \Cont_0(X_i,M)\) to
  \(\Cont_0(X_i, \dvr) \haotimes N \cong \Cont_0(X_i,N)\) by
  \cite{Cortinas-Meyer-Mukherjee:NAHA}*{Proposition~2.4.5}.  The
  embedding property is preserved by taking inductive limits.  Since
  the structure maps of the inductive systems \((\Cont_0(X_i,M))_i\)
  and \((\Cont_0(X_i,N))_i\) are injective, the inductive limit is
  already separated, so that no separated quotient occurs.  Since
  the completed projective tensor product commutes with filtered
  colimits, we obtain the desired bornological embedding from
  \(D \hot M \cong \varinjlim_i \Cont_0(X_i, M)\) into
  \(\varinjlim_i \Cont_0(X_i,N) \cong D \hot N\).
\end{proof}


\section{The definition of analytic cyclic homology}
\label{sec:definition_HA}

In this section, we define analytic cyclic homology for
\(\resf\)\nb-algebras.  The definition is similar to that
in~\cite{Cortinas-Meyer-Mukherjee:NAHA}.  Namely, we lift an
\(\resf\)\nb-algebra~\(A\) to a suitable tensor algebra
over~\(\dvr\).  Then we enlarge this to a tube algebra, put on a
linear growth bornology, tensor with~\(\dvf\), take the
\(X\)\nb-complex and complete in a certain way.  A crucial
difference is that we use the completion for inductive systems of
seminormed \(\dvf\)\nb-vector spaces and not for bornological
\(\dvf\)\nb-vector spaces.  We shift this completion to the very end of
the construction to stay as long as possible with bornologies, which
we find handier than inductive systems.  The tube algebra
construction adds another layer of complexity because it produces
projective systems of algebras.  Thus the end result of our
construction is a \(\Z/2\)\nb-graded chain complex \(\HAC(A)\) over
the additive category of projective systems of inductive systems of
Banach \(\dvf\)\nb-vector spaces.  This additive category is
quasi-Abelian.  We view \(\HAC(A)\) as an object in the resulting
derived category in order to define a bivariant analytic cyclic
homology theory (following the definition by Cuntz and
Quillen~\cite{Cuntz-Quillen:Cyclic_nonsingularity}).  To define
analytic cyclic homology, we form a homotopy projective limit to
remove the projective system layer, then take an inductive limit and
take homology in the usual sense to get a \(\Z/2\)\nb-graded
\(\dvf\)\nb-vector space \(\HA_*(A)\).

We now define our analytic cyclic homology theory through a sequence
of steps as in \cite{Cortinas-Meyer-Mukherjee:NAHA}*{Section~3}.
Let~\(A\) be an \(\resf\)\nb-algebra.  We often view it as a
\(\dvr\)\nb-algebra with \(\dvgen \cdot A = 0\).

\begin{step}[Torsion-free lifting]
  \label{step:torsionfree_lifting}
  Let \(\dvr[A]\) be the free \(\dvr\)\nb-module generated by the
  underlying set of~\(A\).  Equip~\(\dvr[A]\) with the fine
  bornology.  By definition, an element of~\(\dvr[A]\) is a function
  \(h\colon A\to\dvr\) with finite support.  A subset of~\(\dvr[A]\)
  is bounded if and only if it is contained in \(\dvr[X]\) for a
  finite subset \(X\subseteq A\).
  Let~\(M\) be any bornological \(\dvr\)\nb-module.  Any
  \(\dvr\)\nb-module map \(\dvr[A] \to M\) is bounded
  because~\(\dvr[A]\) carries the fine bornology.  Such maps are in
  natural bijection with maps \(A\to M\): the unique bounded
  \(\dvr\)\nb-linear map \(f^\#\colon \dvr[A] \to M\) induced by
  \(f\colon A\to M\) is defined by
  \(f^\#(h) \defeq \sum_{a\in A} h(a) f(a)\).  We let
  \(\varrho\defeq \id_A^\#\colon \dvr[A] \to A\) be the
  \(\dvr\)\nb-linear map
  induced by the identity map on~\(A\).  This map is
  surjective.  Since \(\dvr[A]\) is free as a \(\dvr\)\nb-module and
  carries the fine bornology, it is complete and (bornologically)
  torsionfree.
\end{step}

\begin{step}[Tensor algebras]
  \label{step:tensor_algebra}
  Abbreviate \(W\defeq \dvr[A]\).  Now we replace~\(W\) by a
  bornological \(\dvr\)\nb-algebra, namely, the tensor
  algebra~\(\tens W\).  We recall its definition in the relevant
  category of bornological \(\dvr\)\nb-modules.

  \begin{definition}[\cite{Cortinas-Meyer-Mukherjee:NAHA}*{Lemma~2.6.1}]
    \label{def:tensor_algebra}
    Let~\(W\) be a bornological \(\dvr\)\nb-module.  Its (incomplete,
    non-unital) \emph{tensor algebra}~\(\tens W\) is the direct sum
    \(\bigoplus_{n=1}^\infty W^{\otimes n}\) with its canonical
    bornology and with the multiplication defined by
    \[
      (x_1 \otimes \dotsb \otimes x_n)\cdot (x_{n+1} \otimes \dotsb
      \otimes x_{n+m}) \defeq x_1 \otimes \dotsb \otimes x_{n+m}.
    \]
    Let \(\sigma_W\colon W\to \tens W\) be the inclusion of the first
    summand.  It is a bounded \(\dvr\)\nb-module homomorphism.
  \end{definition}

  A subset \(S\subseteq \tens W\) is bounded if and only if there
  are \(a\in\N\) and a bounded \(\dvr\)\nb-submodule
  \(B\subseteq W\) such that~\(S\) is contained in the image of
  \(\sum_{n=1}^a B^{\otimes n}\) in~\(\tens W\).  If~\(W\) carries
  the fine bornology, then so does~\(\tens W\).  The tensor
  algebra~\(\tens W\) is torsionfree because~\(W\) is torsionfree
  and this is inherited by tensor products and direct sums (see
  \cite{Meyer-Mukherjee:Bornological_tf}*{Proposition~4.12}).  It
  has the following universal property:

  \begin{lemma}
    \label{lem:tensor_algebra}
    Let~\(S\) be a bornological \(\dvr\)\nb-algebra.  For any bounded
    \(\dvr\)\nb-module homomorphism \(f\colon W\to S\), there is a
    unique bounded \(\dvr\)\nb-algebra homomorphism
    \(f^\#\colon \tens W \to S\) with \(f^\# \circ \sigma_W = f\).
  \end{lemma}

  \begin{proof}
    The only possible formula for the homomorphism~\(f^\#\) is
    \[
      f^\#(x_1 \otimes \dotsb \otimes x_n)
      \defeq f(x_1) \dotsm f(x_n)
    \]
    for all \(n\in\N\), \(x_1,\dotsc,x_n \in W\).  And indeed, this
    defines a bounded \(\dvr\)\nb-algebra homomorphism
    \(f^\#\colon \tens W \to S\) with \(f^\# \circ \sigma_W = f\).
  \end{proof}

  For \(W=\dvr[A]\), the universal properties of tensor algebras
  and free modules together show that (bounded) \(\dvr\)\nb-algebra
  homomorphisms \(\tens \dvr[A] \to S\) for a bornological
  \(\dvr\)\nb-algebra~\(S\) are naturally in bijection with maps
  \(A\to S\).  Thus \(\tens \dvr[A]\) is the free algebra on the
  set~\(A\), equipped with the fine bornology.
  Since~\(A\) with the fine bornology is a bornological
  \(\dvr\)\nb-algebra, \(\varrho = \id_A^\#\colon \dvr[A] \to A\)
  induces a bounded
  \(\dvr\)\nb-algebra homomorphism \(\varrho^\#\colon \tens W\onto A\)
  by Lemma~\ref{lem:tensor_algebra}.  Let \(I\defeq \ker \varrho^\#\).
  There is a bornological \(\dvr\)\nb-algebra resolution
  \(I \into \tens W \onto A\).
\end{step}

\begin{step}[Tube algebras]
  \label{step:tube_algebra}
  So far, we have built a torsionfree bornological
  \(\dvr\)\nb-algebra \(R\defeq \tens \dvr[A]\) with an ideal~\(I\)
  and a bornological isomorphism \(R/I \cong A\).  Now we
  enlarge~\(R\) to the tube algebras over powers of the ideal~\(I\).
  We recall their definition:

  \begin{definition}[\cite{Cortinas-Meyer-Mukherjee:NAHA}*{Definition~3.1}]
    \label{def:tube_algebra}
    Let~\(R\) be a torsionfree bornological \(\dvr\)\nb-algebra
    and~\(I\) an ideal in~\(R\).  Let~\(I^j\) for \(j\in\N^*\) denote
    the \(\dvr\)\nb-linear span of products \(x_1\dotsm x_j\) with
    \(x_1,\dotsc,x_j\in I\).  The \emph{tube algebra of
      \(I^l\idealin R\)} for \(l\in\N^*\) is defined as
    \[
      \tub{R}{I^l}\defeq \sum_{j=0}^{\infty} \dvgen^{-j}I^{l\cdot j}
      \subseteq R \otimes \dvf,
    \]
    equipped with the subspace bornology; this is indeed a
    \(\dvr\)\nb-subalgebra of \(R \otimes \dvf\).  If \(l \ge j\),
    then \(\tub{R}{I^l} \subseteq \tub{R}{I^j}\), and the inclusion
    is a bounded algebra homomorphism.  Let \(\tub{R}{I^\infty}\)
    denote the resulting projective system of bornological
    \(\dvr\)\nb-algebras \((\tub{R}{I^l})_{l\in\N^*}\).
  \end{definition}

  Since \(\tub{R}{I^l}\) is a bornological submodule of an
  \(\dvf\)\nb-vector space, it is torsionfree (see
  \cite{Meyer-Mukherjee:Bornological_tf}*{Proposition~4.3}).  The
  inclusion \(R \hookrightarrow \tub{R}{I^l}\) induces bornological
  isomorphisms
  \[
    \tub{R}{I^l} \otimes \dvf \cong R \otimes \dvf
    \qquad \text{for all }l\in\N.
  \]
  In fact, for \(R=\tens \dvr[A]\), the bornology on
  \(\tub{R}{I^l}\) is the fine one, just as on \(\dvr[A]\) and~\(R\).
  We allow more general bornologies in the previous steps because
  this will happen in later generalisations.
\end{step}

\begin{step}[Linear growth bornology]
  \label{step:linear_growth}
  Equip each tube algebra \(\tub{R}{I^l}\) for \(l \in \N^*\) with
  the linear growth bornology
  (see~\cite{Cortinas-Cuntz-Meyer-Tamme:Nonarchimedean}); we denote
  this new bornological \(\dvr\)\nb-algebra by
  \(\ling{\tub{R}{I^l}}\).  It is a semidagger algebra by
  construction and remains torsionfree (see
  \cite{Meyer-Mukherjee:Bornological_tf}*{Definition 3.3 and
    Proposition~4.11}).  Let \(\ling{\tub{R}{I^\infty}}\) be the
  projective system formed by the torsionfree, semidagger
  bornological \(\dvr\)\nb-algebras \(\ling{\tub{R}{I^l}}\) for
  \(l\in\N^*\).
\end{step}

\begin{step}[Tensor with the field of fractions]
  \label{step:tensor_dvf}
  Tensor with~\(\dvf\).  This gives a projective system of
  bornological \(\dvf\)\nb-algebras
  \(\ling{\tub{R}{I^\infty}} \otimes \dvf
    \defeq (\ling{\tub{R}{I^l}} \otimes
    \dvf)_{l\in\N^*}\).
\end{step}

\begin{step}[$X$-Complex]
  \label{step:X}
  Form the \emph{\(X\)\nb-complex}, which gives us a
  projective system of \(\Z/2\)\nb-graded chain complexes of
  bornological \(\dvf\)\nb-vector spaces
  \[
    X\bigl(\ling{\tub{R}{I^\infty}} \otimes \dvf\bigr)
    \defeq X\bigl(\ling{\tub{R}{I^l}} \otimes
    \dvf\bigr)_{l\in\N^*}.
  \]
  We recall how the \(X\)\nb-complex is defined for incomplete
  bornological algebras:

  \begin{definition}
    \label{def:incomplete_X}
    Let~\(S\) be a bornological \(\dvr\)\nb-algebra.  Let
    \(\Omega^1(S) \defeq S^+ \otimes S\) be the bornological
    \(S\)\nb-bimodule of
    noncommutative \(1\)\nb-forms.  Let
    \(\Omega^1(S)/[\cdot, \cdot]\) be the quotient of
    \(\Omega^1(S)\) by the image of the commutator map
    \[
      S \otimes \Omega^1(S) \to \Omega^1(S),\qquad
      x \otimes \omega \mapsto x \cdot \omega - \omega \cdot x.
    \]
    The \emph{\(X\)\nb-complex} of~\(S\) is defined as
    \[
      X(S)\defeq
      \Bigl(
      \begin{tikzcd}
        S \arrow[r, shift left, "q\circ \diff"]&
        \Omega^1(S)/[\cdot,\cdot] \arrow[l, shift left, "\tilde{b}"]
      \end{tikzcd}
      \Bigr).
    \]
    The maps \(q\), \(\diff\) and~\(\tilde{b}\) are the same maps as
    in~\cite{Cuntz-Quillen:Cyclic_nonsingularity} or
    \cite{Cortinas-Meyer-Mukherjee:NAHA}*{Section~2.7}.

    The order
    of Steps \ref{step:tensor_dvf} and~\ref{step:X} may be reversed,
    that is, \(X(S \otimes \dvf) \cong X(S) \otimes \dvf\).
  \end{definition}
\end{step}

\begin{step}[Quasi-completion] Pass to inductive systems and then
  complete there.
  \label{step:completion}

  \begin{definition}
    \label{def:HA_complex}
    The \emph{analytic cyclic homology complex} of~\(A\) is
    \[
      \HAC(A) \defeq
      \comb{\diss (X(\ling{\tub{R}{I^\infty}})\otimes \dvf)}
    \]
    with \(R\defeq \tens \dvr[A]\) and
    \(I\defeq \ker(\id_A^\#\colon \tens \dvr[A] \onto A)\); this is a
    \(\Z/2\)\nb-graded chain complex of projective systems of
    inductive systems of Banach \(\dvf\)\nb-vector spaces; we
    briefly call this a \emph{pro-ind complex}.
  \end{definition}
\end{step}

The following steps extract a homology \(\HA_*(A)\) from the pro-ind
complex \(\HAC(A)\).

\begin{step}[Homotopy limit]
  \label{step:htpy_lim}
  Let \(X=(X_n)_{n\in\N}\) with \(X_n\) in
  \(\mathsf{Ind}(\mathsf{Ban}_\dvf)\) for \(n\in\N\) be a pro-ind
  complex.  Then we let \(\qholim (X)\) denote the pro-ind complex
  \[
    \qholim(X)
    = \cone\bigl(\fake X_n \xrightarrow{\id - \sigma}
    \fake X_n\bigr),
  \]
  where \(\cone\) denotes the mapping cone, \(\sigma\)~is the shift
  map that
  restricts to the structure map \(X_{n+1} \to X_n\) on the factors,
  and the product is a formal one in the category of projective
  systems \(\overleftarrow{\mathsf{Ind}(\mathsf{Ban}_\dvf)}\).
\end{step}

The composite of the canonical inclusion \(X \to \fake X_n\) with
the canonical map \(\fake X_n \to \qholim X\) is a quasi-isomorphism
between \(X\) and \(\qholim X\) (see ~\cite{Neeman:Triangulated}*{Lemma 1.6.6}, for example).  In particular,
\(\qholim \HAC(A)\) is quasi-isomorphic to \(\HAC(A)\).  As a
consequence, two isomorphic pro-ind complexes have quasi-isomorphic
homotopy limits.  In fact, their homotopy limits are even chain
homotopy equivalent (see~\cite{Neeman:Triangulated}*{1.7.1} for a closely related statement).  They may, however, fail to be
isomorphic as chain complexes.  The pro-ind complex
\(\qholim \HAC(A)\) depends functorially on the algebra~\(A\)
because an algebra homomorphism induces a morphism of diagrams, not
merely of projective systems, and the homotopy limit construction is
clearly functorial for diagram morphisms.

\begin{step}[Homology]
  \label{step:homology}
  The homology of a pro-ind complex~\(X\) is defined most quickly as
  the \(\dvf\)\nb-vector space of chain homotopy classes of chain
  maps \(\dvf \to X\), where we treat~\(\dvf\) as a constant pro-ind
  complex.  More explicitly, this means the following.  We first
  take levelwise inductive limits and forget the bornology, getting
  a countable projective system of chain complexes of
  \(\dvf\)\nb-vector spaces.  Secondly, we take the projective
  limit, getting a chain complex of \(\dvf\)\nb-vector spaces.
  Third, we take its homology.
\end{step}

\begin{definition}
  The \emph{analytic cyclic homology} \(\HA_*(A)\) of an
  \(\resf\)\nb-algebra~\(A\) is the homology as explained above of
  the chain complex of \(\dvf\)\nb-vector spaces
  \(\qholim \HAC(A)\).  In other words, it is the homology of the
  chain complex of \(\dvf\)\nb-vector spaces
  \(\varprojlim \varinjlim \qholim \HAC(A)\).
\end{definition}

Thus \(\HA_*(A)\) for \(* = 0,1\) are \(\dvf\)\nb-vector spaces
without further structure.  (We prefer to ignore the natural
bornology on \(\HA_*(A)\) because we have no good use for it.  If we
cared about this bornology, we should have taken a homotopy
inductive limit instead of an ordinary inductive limit before taking
homology because inductive limits need not be exact on the level of
bornologies.)

We may slightly simplify the construction above.  Taking the
projective limit after the homotopy inverse limit above has the
effect of replacing the formal product \(\fake\) by an ordinary
product in the category of \(\dvf\)\nb-vector spaces.  This gives
the homotopy projective limit
\(\holim = \varprojlim \circ \qholim\).  It is important to apply
the \(\holim\) construction after the levelwise inductive limit.  We
thank the referee for pointing out that we did not treat these limit
procedures correctly in a previous version of the article.



\subsection{Bivariant analytic cyclic homology}

Now we define bivariant analytic cyclic homology, by taking maps
\(\HAC(A)\to \HAC(B)\) in the derived category of
\(\overleftarrow{\mathsf{Ind}(\mathsf{Ban}_\dvf)}\), the category of
countable projective systems of inductive systems of Banach
\(\dvf\)\nb-vector spaces.  We show that this category is
quasi-Abelian, so that it has a derived category.

We recall some basic concepts from
Schneiders~\cite{Schneiders:Quasi-Abelian}.  Let~\(\mathcal{C}\) be
an additive category with kernels and cokernels.
A diagram \(K \overset{i}\to E \overset{p}\to Q\) is an
\emph{extension} if~\(i\) is a kernel of~\(p\) and~\(p\) is a
cokernel of~\(i\).  An additive category with kernels and cokernels
is \emph{quasi-Abelian} if the pullback of a cokernel along
any map exists and is again a cokernel and the pushout of a kernel
along any map exists and is again a kernel.
Our example of interest is
\(\overleftarrow{\mathsf{Ind}(\mathsf{Ban}_\dvf)}\).

Throughout this article, we restrict ourselves to
\(\Z/2\)-graded chain complexes, also known as
\emph{supercomplexes}. More concretely, these are chain complexes with \(2\)-periodic entries, and morphisms between such complexes is given by the \(\Z/2\)-graded mapping complex defined in \cite{Meyer:HLHA}*{A.4, A.10}. A chain complex \((C,\delta)\)
in a quasi-Abelian category~\(\mathfrak{B}\) is \emph{exact} if
\(\ker(\delta)\) exists and \(\delta\colon C \to \ker(\delta)\) is a
cokernel.  A chain map \(f \colon C \to D\) is a
\emph{quasi-isomorphism} if its mapping cone is exact.  The
\emph{derived category} \(\mathrm{Der}(\mathfrak{B})\)
of~\(\mathfrak{B}\) is the localisation of the triangulated homotopy category of
(\(\Z/2\)-graded)
chain complexes at the quasi-isomorphisms.  Equivalently, it is the
quotient of the homotopy category by the thick subcategory of exact
chain complexes. 



\begin{lemma}
  The categories \(\mathsf{Ban}_\dvf\),
  \(\mathsf{Ind}(\mathsf{Ban}_\dvf)\) and
  \(\overleftarrow{\mathsf{Ind}(\mathsf{Ban}_\dvf)}\) are
  quasi-Abelian categories.
\end{lemma}

\begin{proof}
  The categories \(\mathsf{Ban}_\dvf\) and
  \(\mathsf{Ind}(\mathsf{Ban}_\dvf)\) are quasi-Abelian by
  \cite{Bambozzi-Ben-Bassat:Dagger}*{Proposition~3.18} and
  \cite{Bambozzi-Ben-Bassat:Dagger}*{Proposition 3.29},
  respectively.  By the same proof as in
  \cite{Prosmans:Derived_limits}*{Proposition~7.1.5}, 
  \(\overleftarrow{\mathsf{Ind}(\mathsf{Ban}_\dvf)}\) is
  quasi-Abelian. \qedhere  
\end{proof}

\begin{definition}
  Let \(A\) and~\(B\) be \(\resf\)\nb-algebras.  Then \(\HAC(A)\)
  and \(\HAC(B)\) are \(\Z/2\)\nb-graded chain complexes over the
  \(\dvf\)\nb-linear quasi-Abelian category
  \(\overleftarrow{\mathsf{Ind}(\mathsf{Ban}_\dvf)}\).  Let
  \(\Hom_{\mathsf{Der}(\overleftarrow{\mathsf{Ind}(\mathsf{Ban}_\dvf)})}\)
  denote morphisms in the derived category of \(\Z/2\)\nb-graded
  chain complexes.  Let~\([i]\) denote a degree shift by~\(i\).  The
  \emph{bivariant analytic cyclic homology} \(\HA_*(A,B)\) is the
  \(\Z/2\)\nb-graded \(\dvf\)\nb-vector space with
  \[
    \HA_i(A,B) \defeq
    \Hom_{\mathsf{Der}(\overleftarrow{\mathsf{Ind}(\mathsf{Ban}_\dvf)})}
    (\HAC(A), \HAC(B)[i])
    \qquad \text{for }i=0,1.
  \]
\end{definition}

\begin{remark}
To clarify the mapping space of the above derived category, the authors in (\cite{mukherjee2022quillen}) describe a Quillen model structure on the category of \(\Z/2\)-graded chain complexes over \(\overleftarrow{\mathsf{Ind}(\mathsf{Ban}_\dvf)}\), whose homotopy category is the derived category of the quasi-abelian category \(\overleftarrow{\mathsf{Ind}(\mathsf{Ban}_\dvf)}\). This is based on previous work by Corti\~nas and Valqui (\cite{Cortinas-Valqui:Excision}), who define a homotopy category of pro-\(\Z/2\)-graded chain complexes to describe the mapping space for bivariant periodic cyclic homology.     
\end{remark}

The following theorem uses some results that will be proven only
later.  We put it here to point out right away how
\(\HA_*(A,B)\) generalises \(\HA_*(B)\).

\begin{theorem}
  \label{theorem:bivariant_specialises}
  If \(A = \resf\), then \(\HA_*(\resf, B) \cong \HA_*(B)\).
\end{theorem}

\begin{proof}
  Since~\(\dvr\) is a dagger algebra that lifts~\(\resf\), we may
  use it to compute the analytic cyclic homology of~\(\resf\) by
  Corollary~\ref{cor:Monsky-Washnitzer_lift} and
  Theorem~\ref{the:HA_through_lift}.  That is, there is a
  quasi-isomorphism \(\HAC(\resf) \simeq \diss \HAC(\dvr)\) with the
  complex \(\HAC(\dvr)\) defined
  in~\cite{Cortinas-Meyer-Mukherjee:NAHA}.  By
  \cite{Cortinas-Meyer-Mukherjee:NAHA}*{Corollary 4.7.3},
  \(\HAC(\dvr)\) is chain homotopy equivalent to~\(\dvf\), supported
  in degree~\(0\).  Here~\(\dvf\) is viewed as a constant pro-object in \(\mathsf{CBor}_\dvf\).  The dissection functor
  \(\mathsf{CBor}_\dvf \to
  \overleftarrow{\mathsf{Ind}(\mathsf{Ban}_\dvf})\) is additive and
  therefore descends to a functor between the homotopy categories of
  complexes.  So by functoriality,
  \(\diss (\HAC(\dvr)) \simeq \dvf\).  It remains to compute the
  morphism space
  \begin{equation}
    \label{mapping_space}
    \Hom_{\mathsf{Der}(\overleftarrow{\mathsf{Ind}(\mathsf{Ban}_\dvf)})}
    (\dvf, \HAC(B)[i]).
  \end{equation}
  By definition of the derived category of a quasi-abelian category,
  morphisms between \(\dvf\) and \(\HAC(B)[i]\) are given by
  equivalence classes of diagrams of the form
  \[
    \begin{tikzcd}
      \dvf \arrow{dr} & & \HAC(B)[i], \ar[dl, "\simeq"]\\
      & C &
    \end{tikzcd}
  \]
  where \(\HAC(B)[i] \overset{\sim}\to C\) is a quasi-isomorphism to
  an object ~\(C\) of
  \(\mathsf{Der}(\overleftarrow{\mathsf{Ind}(\mathsf{Ban}_\dvf)})\).
  Here the arrows are homotopy classes of chain maps.  The set of
  arrows \(F\to C\) is exactly the homology of the chain complex
  \(\varprojlim \varinjlim C\).  The proof will be finished by
  showing that any zigzag as above may be replaced by one where the
  map \(\HAC(B)[i] \to C\) is the canonical map to
  \(\qholim \HAC(B)[i]\); this fact is the reason why we built the
  homotopy limit before taking homology in the definition of
  \(\HA_*(A)\).

  Given a zigzag as above, we get the solid arrows in the following
  diagram:
  \[
    \begin{tikzcd}
      \dvf \arrow{dr} \ar[drrr, dotted] & &
      \HAC(B)[i] \ar[dl, "\simeq"] \ar[dr, "\simeq"] & &\\
      & C \ar[dr, "\simeq"] & &
      \qholim \HAC(B)[i] \ar[dl, "\simeq"] \\
      & & \qholim (C) & &
    \end{tikzcd}
  \]
  Here we use the weak functoriality of homotopy limits in
  triangulated categories to get the induced quasi-isomorphism
  \(\qholim \HAC(B)[i] \to \qholim C\).  Since both products and
  inductive limits in the category of \(\dvf\)\nb-vector spaces are
  exact functors, it follows that the quasi-isomorphism
  \(\qholim \HAC(B)[i] \to \qholim C\) induces an isomorphism on the
  homology of these pro-ind complexes; the point here is that the
  range of \(\qholim\) consists of chain complexes where the
  projective limit is just a product.  Therefore, the arrow
  \(\dvf \to \qholim C\) factors uniquely (up to chain homotopy)
  through the dashed arrow \(\dvf \to \qholim \HAC(B)[i]\).  Thus
  any arrow \(\dvf \to \HAC(B)[i]\) in the derived category is
  represented by a zigzag
  \[
    \dvf \to \qholim \HAC(B)[i] \overset{\simeq}\leftarrow \HAC(B)[i].
  \]
  These special zigzags are in bijection with
  \(H_i(\holim \varinjlim \HAC(B)) \cong \HA_i(B)\).  This implies
  that there is a well defined surjective map
  \(\HA_i(B) \to \HA_i(\resf,B)\).

  Two zigzags \(\dvf \to C_j \leftarrow \HAC(B)[i]\) for \(j=1,2\)
  as above represent the same arrow in the derived category if and
  only if we may ``enlarge'' them by chain maps \(C_1\to C\) and
  \(C_2 \to C\) to the same chain complex~\(C\), so that the chain
  maps \(\dvf \to C\) and \(\HAC(B)[i] \to C\) become chain
  homotopic.  Then the zigzags through \(C_1\), \(C_2\) and~\(C\)
  all correspond to the same homology class in
  \(H_i( \holim \varinjlim \HAC(B))\).  Thus the surjective map
  \(\HA_i(B) \to \HA_i(\resf,B)\) is also injective.
\end{proof}

\section{A variant of the definition}
\label{sec:homotopy_stability}

Let~\(A\) be an \(\resf\)\nb-algebra and let~\(D\) be a dagger
algebra over~\(\dvr\) with \(D/\dvgen D \cong A\) bornologically,
that is, \(D/\dvgen D\) carries the fine bornology.  This situation
occurs, for instance, in the definition of Monsky--Washnitzer
cohomology for a smooth commutative \(\resf\)\nb-algebra~\(A\).  Let
\(\HAC(D)\) be the analytic cyclic homology complex of the dagger
algebra~\(D\) as defined in~\cite{Cortinas-Meyer-Mukherjee:NAHA}.
One of our main results shows that \(\HAC(A)\) is quasi-isomorphic
to \(\diss \HAC(D)\) with the dissection functor
in~\eqref{eq:dissection_over_F_complete}.  This implies
\(\HA_*(D) \cong \HA_*(A)\).  To prove this, we identify both
\(\diss \HAC(D)\) and \(\HAC(A)\) with an auxiliary invariant that
depends on \(A\) and~\(D\).  In this section, we define this invariant
and formulate its functoriality and homotopy invariance properties.

We do not know whether any \(\resf\)\nb-algebra~\(A\) is isomorphic
to a quotient \(D/\dvgen D\) for a dagger algebra~\(D\).  What we do
know is that there is an epimorphism \(\varrho\colon D \onto A\)
where \(D=(D_n)_{n\in\N}\) is a pro-dagger algebra and
\(\ker \varrho\) is analytically nilpotent; in addition, we may
arrange for \(D_n/\dvgen D_n\) to carry the fine bornology for all
\(n\in\N\).  We prove later that \(\HAC(A) \simeq \diss \HAC(D)\)
remains true in this more general situation.  This allow us to carry
over results about the analytic cyclic homology of pro-dagger
algebras from~\cite{Cortinas-Meyer-Mukherjee:NAHA} to our new
theory.

The proof of this quasi-isomorphism uses a variant of the definition
above that depends, in addition, on a
projective system of torsionfree complete
bornological \(\dvr\)\nb-modules \(W = (W_n)_{n\in\N}\) and a
morphism of projective systems of bornological \(\dvr\)\nb-modules
\(\varrho\colon W \to A\); here~\(A\) carries the fine bornology and
is viewed as a constant projective system.
The morphism~\(\varrho\) is represented by a bounded
\(\dvr\)\nb-linear map \(\varrho_n\colon W_n \to A\) for some
\(n\in\N\).  Reindexing the projective system, we may arrange that
\(n=0\).  In particular, if~\(W\) is a constant projective system,
we are dealing with a torsionfree complete bornological
\(\dvr\)\nb-module~\(W\) and a bounded \(\dvr\)\nb-linear map
\(\varrho\colon W\to A\).

In Step~\ref{step:torsionfree_lifting} of the construction above, we
now replace~\(\dvr[A]\) by~\(W\) and the canonical map
\(\id_A^\#\colon \dvr[A] \to A\) by the given map
\(\varrho\colon W\to A\).  The remaining steps are the same as
above, applied levelwise to projective systems of algebras.  In
Step~\ref{step:tensor_algebra}, we define the tensor algebra of a
projective system levelwise; this has an analogous universal
property.  The kernel~\(I\) of the induced map
\(\varrho^\#\colon \tens W \to A\) is computed levelwise as the
projective system formed by
\(I_n \defeq \ker (\varrho_n^\#\colon \tens W_n \to A)\) for
\(n\in\N\).  In Step~\ref{step:tube_algebra}, we form the levelwise
tube algebra
\(\tub{\tens W}{I^\infty} \defeq (\tub{\tens W_n}{I^l})_{n\in\N,
  l\in\N^*}\) as in
\cite{Cortinas-Meyer-Mukherjee:NAHA}*{Section~4.2}.  In the
following steps, \ref{step:linear_growth}, \ref{step:tensor_dvf},
\ref{step:X} and~\ref{step:completion}, we apply the same
construction as above levelwise.  This results in the
\(\Z/2\)\nb-graded chain complex of projective systems of
inductive systems of Banach \(\dvf\)\nb-vector spaces
\[
  \HAC(A,W,\varrho) \defeq
  \Bigl(\comb{\diss X( \ling{\tub{\tens W_n}{I^l}} \otimes \dvf)}
  \Bigr)_{n\in\N, l\in\N^*}.
\]
The definitions above are consistent in the following sense:
\begin{align*}
  \HAC(A) &\defeq \HAC(A, \dvr[A], \id_A^\#\colon \dvr[A] \onto A),\\
  \HA_*(A) &\defeq \HA_*(A, \dvr[A], \id_A^\#\colon \dvr[A] \onto A).
\end{align*}

The following lemma explains the functoriality of
\(\HA(A,W,\varrho)\):

\begin{lemma}
  \label{lem:hybrid_functorial}
  Let~\(A_j\) for \(j=1,2\) be \(\resf\)\nb-algebras, viewed as
  constant projective systems of \(\dvr\)\nb-algebras, and let
  \(f\colon A_1 \to A_2\) be an \(\resf\)\nb-algebra homomorphism.
  Let \(W_j = (W_{j,n})_{n\in\N}\) for \(j=1,2\) be projective
  systems of torsionfree complete bornological \(\dvr\)\nb-modules
  and let \(\varrho_j\colon W_j \to A_j\) and
  \(g\colon W_1 \onto W_2\) be morphisms of projective systems of
  bornological \(\dvr\)\nb-modules.  Assume that the square
  \begin{equation}
    \label{eq:HA_W_functorial}
    \begin{tikzcd}
      W_1 \arrow[d, "\varrho_1"'] \arrow[r, "g"] &
      W_2 \arrow[d, "\varrho_2"] \\
      A_1 \arrow[r, "f"] &
      A_2
    \end{tikzcd}
  \end{equation}
  commutes.  Let
  \(I_j \defeq \ker (\varrho_j^\#\colon \tens W_j\to A_j)\).  There
  is an induced compatible family of morphisms of projective systems
  of \(\dvf\)\nb-algebras
  \[
    (g,f)_*\colon \ling{\tub{\tens W_1}{I_1^l}} \otimes \dvf\to
    \ling{\tub{\tens W_2}{I_2^l}} \otimes \dvf
  \]
  for all \(l\in\N\), which further induces a chain map
  \[
    (g,f)_*\colon \HAC(A_1,W_1,\varrho_1) \to \HAC(A_2,W_2,\varrho_2),
  \]
  which further induces a grading-preserving \(\dvf\)\nb-linear map
  \[
    (g,f)_*\colon \HA_*(A_1, W_1,\varrho_1)
    \to \HA_*(A_2, W_2,\varrho_2).
  \]
  All these constructions are functorial for the obvious
  composition of commuting squares as in~\eqref{eq:HA_W_functorial}.
\end{lemma}

\begin{proof}
  Reindexing the projective systems \(W_1\) and~\(W_2\), we may
  arrange that \(g\) and~\(\varrho_j\) for \(j=1,2\) are represented
  by a morphism of diagrams \(g_n\colon W_{1,n} \to W_{2,n}\) for
  \(n\in\N\) and a compatible family of bounded \(\dvr\)\nb-linear
  maps \(\varrho_{j,n} \colon W_{j,n} \onto A_j\) for all
  \(n\in\N\).  The commuting square~\eqref{eq:HA_W_functorial} says
  that \(f\circ \varrho_{1,n} = \varrho_{2,n} \circ g_n\) holds for
  sufficiently large~\(n\), and we may reindex so that this holds
  for all \(n\in\N\).  We assume all this to simplify the notation.
  By Lemma~\ref{lem:tensor_algebra}, the bounded linear maps
  \(\varrho_{j,n}\) and~\(g_n\) induce bounded \(\dvr\)\nb-algebra
  homomorphisms \(\varrho_{j,n}^\# \colon \tens W_{j,n} \to A_j\)
  and \(\tens(g_n) \colon \tens W_{1,n} \to \tens W_{2,n}\) with
  \(\tens(g_n) \circ \sigma_{W_{1,n}} = \sigma_{W_{2,n}}\circ g_n\)
  and \(\varrho_{j,n} = \varrho_{j,n}^\# \circ \sigma_{W_{j,n}}\).
  We compute
  \[
    \varrho_{2,n}^\# \circ \tens(g_n)\circ \sigma_{W_{1,n}}
    = \varrho_{2,n}^\# \circ \sigma_{W_{2,n}}\circ g_n
    = \varrho_{2,n}\circ g_n
    = f \circ \varrho_{1,n}
    = f \circ \varrho_{1,n}^\# \circ \sigma_{W_{1,n}}.
  \]
  The uniqueness part of Lemma~\ref{lem:tensor_algebra} implies
  \(\varrho_{2,n}^\# \circ \tens(g_n) = f \circ \varrho_{1,n}^\#\).
  Then~\(\tens(g_n)\) maps~\(I_{1,n}\) to~\(I_{2,n}\).  Hence
  \(\tens(g_n)\) extends uniquely to bounded homomorphisms
  \(\tub{\tens W_{1,n}}{I_{1,n}^l} \to \tub{\tens
    W_{2,n}}{I_{2,n}^l}\) for all \(l\in\N\).  These homomorphisms
  remain bounded in the linear growth bornologies because the latter
  is functorial for bounded homomorphisms.  The tensor product
  with~\(\dvf\), the \(X\)\nb-complex and the quasi-completion are
  functors as well.  Thus we get an induced chain map
  \[
    (g,f)_*\colon \HAC(A_1, W_1, \varrho_1) \to
    \HAC(A_2, W_2, \varrho_2).
  \]
  Here ``chain map'' means a morphism of \(\Z/2\)\nb-graded chain
  complexes of projective systems of inductive systems of Banach
  \(\dvf\)\nb-vector spaces.  Since homotopy projective limits,
  inductive limits, and homology are functorial as well, this
  induces a grading-preserving \(\dvf\)\nb-linear map
  \((g,f)_*\colon \HA_*(A_1, W_1, \varrho_1) \to \HA_*(A_2, W_2,
  \varrho_2)\).  It is easy to see that this construction is a
  functor.
\end{proof}

We have highlighted the maps
\(\ling{\tub{\tens W_1}{I_1^l}} \otimes \dvf\to \ling{\tub{\tens
    W_2}{I_2^l}} \otimes \dvf\) because they are what we will use in
later proofs.  We shall also need the following homotopy invariance
statement.  Actually, we will only need the special case
\(A_1=A_2=A\) and \(f=\id_A\), but this case is only notationally
easier.

\begin{proposition}
  \label{pro:homotopy_lifting_2}
  Let \(A_j\) and \(\varrho_j \colon W_j \onto A_j\) be as in
  Lemma~\textup{\ref{lem:hybrid_functorial}}.  Let
  \(f_0,f_1\colon A_1 \rightrightarrows A_2\) and
  \(H \colon A_1 \to A_2[t]\) be algebra homomorphisms with
  \(f_j = \ev_j\circ H\) for \(j=0,1\); that is, \(H\) is an
  \emph{elementary polynomial homotopy} between \(f_0\) and~\(f_1\).
  Let \(H' \colon W_1 \to W_2 \otimes \dvr[t]\) be a morphism of
  projective systems of bornological \(\dvr\)\nb-modules that makes
  the square
  \[
    \begin{tikzcd}
      W_1 \arrow[d, "\varrho_1"'] \arrow[r, "H'"] &
      W_2 \otimes \dvr[t] \arrow[d, "\varrho_2\otimes \id"] \\
      A_1 \arrow[r, "H"] &
      A_2[t]
    \end{tikzcd}
  \]
  commute.  Then \(g_j \defeq \ev_j\circ H'\colon W_1 \to W_2\) for
  \(j=0,1\) are morphisms of projective systems such that the square
  in~\eqref{eq:HA_W_functorial} commutes for \(f_j\) and~\(g_j\).
  There is a compatible family of morphisms of projective systems of
  bornological \(\dvf\)\nb-algebras
  \[
    (H',H)_*\colon \ling{\tub{\tens W_1}{I_1^l}} \otimes \dvf\to
    \ling{\tub{\tens W_2}{I_2^l}} \otimes \ling{\dvr[t]} \otimes \dvf
  \]
  for all \(l\in\N\), which provides elementary homotopies between
  the maps \((g_0,f_0)_*\) and \((g_1,f_1)_*\) built in
  Lemma~\textup{\ref{lem:hybrid_functorial}}.  The morphism
  \((H',H)_*\) induces a chain homotopy between the chain maps
  \[
    (g_0,,f_0)_*, (g_1,,f_1)_*
    \colon \HAC(A_1,W_1,\varrho_1) \rightrightarrows
    \HAC(A_2,W_2,\varrho_2).
  \]
  This implies an equality of maps on homology:
  \[
    (g_0,,f_0)_* = (g_1,,f_1)_*\colon \HA_*(A_1, W_1,\varrho_1)
    \to \HA_*(A_2, W_2,\varrho_2).
  \]
\end{proposition}

\begin{proof}
  The functoriality of the tensor algebra construction gives a
  unique morphism
  \(\tens(H')\colon \tens W_1 \to \tens (W_2 \otimes \dvr[t])\) with
  \(\tens(H') \circ \sigma_{W_1} = \sigma_{W_2\otimes \dvr[t]} \circ
  H'\).  The morphism
  \[
    \sigma_{W_2}\otimes \id_{\dvr[t]} \colon
    W_2 \otimes \dvr[t] \to (\tens W_2) \otimes \dvr[t]
  \]
  induces a unique morphism
  \[
    (\sigma_{W_2}\otimes \id_{\dvr[t]})^\# \colon
    \tens (W_2 \otimes \dvr[t]) \to (\tens W_2) \otimes \dvr[t]
  \]
  with
  \((\sigma_{W_2}\otimes \id_{\dvr[t]})^\# \circ \sigma_{W_2 \otimes
    \dvr[t]} = \sigma_{W_2}\otimes \id_{\dvr[t]}\).  We let
  \[
    (H',H)_* \defeq (\sigma_{W_2}\otimes \id_{\dvr[t]})^\# \circ
    \tens(H') \colon \tens W_1 \to (\tens W_2) \otimes \dvr[t].
  \]
  Using the universal property of the tensor algebra, we compute
  \begin{align*}
    (\id_{\tens W_2} \otimes \ev_j) \circ (H',H)_*
    &= \tens g_j\colon \tens W_1 \to \tens W_2
    \\ \shortintertext{for \(j=0,1\) and}
    (\varrho_2^\# \otimes \id_{\dvr[t]}) \circ (H',H)_*
    &= H \circ \varrho_1^\# \colon \tens W_1 \to A \otimes_\resf \resf[t].
  \end{align*}
  Therefore, \((H',H)_*\) maps the ideal
  \(I_1\defeq \ker (\varrho_1^\#\colon \tens W_1 \to A)\) to
  \(I_2 \otimes \dvr[t]\).  Thus it extends to a morphism of tube
  algebras
  \[
    \tub{\tens W_1}{I_1^l} \to
    \tub{(\tens W_2) \otimes \dvr[t]}{(I_2 \otimes \dvr[t])^l}
    \cong \tub{\tens W_2}{I_2^l} \otimes \dvr[t].
  \]
  Since the levelwise linear growth bornology is natural, this
  remains a morphism
  \[
    \ling{\tub{\tens W_1}{I_1^l}} \to
    \ling{(\tub{\tens W_2}{I_2^l} \otimes \dvr[t])}.
  \]
  The target is identified with
  \(\ling{\tub{\tens W_2}{I_2^l}} \otimes \ling{\dvr[t]}\) by
  \cite{Cortinas-Cuntz-Meyer-Tamme:Nonarchimedean}*{Proposition~3.1.25}.
  So we get a morphism of projective systems of \(\dvr\)\nb-algebras
  \[
    (H',H)_*\colon \ling{\tub{\tens W_1}{I_1^l}} \to
    \ling{\tub{\tens W_2}{I_2^l}} \otimes \ling{\dvr[t]}.
  \]
  By the computations above, this is an elementary homotopy between
  \((g_0,f_0)_*\) and \((g_1,f_1)_*\).  Tensoring with~\(\dvf\) and
  letting~\(l\) vary gives an elementary homotopy
  \[
    \ling{\tub{\tens W_1}{I_1^\infty}} \otimes \dvf \to
    \ling{\tub{\tens W_2}{I_2^\infty}} \otimes \dvf \otimes
    \ling{\dvr[t]}.
  \]
  The pro-algebras
  \(\ling{\tub{\tens W_j}{I_j^\infty}} \otimes \dvf\) that occur
  here are quasi-free, that is, their bimodules of (incomplete)
  noncommutative differential \(1\)\nb-forms are projective.
  Therefore, an elementary homotopy between two morphisms
  of pro-algebras as above induces a chain homotopy between the
  induced chain maps on the \(X\)\nb-complexes.  This is shown as in
  the proof of
  \cite{Cortinas-Meyer-Mukherjee:NAHA}*{Proposition~4.6.1}, leaving
  out completions everywhere.  Since the quasi-completion is an
  additive functor, this chain homotopy passes on to the chain
  complexes \(\HAC(A_j,W_j,\varrho_j)\) for \(j=1,2\).  Then
  the induced maps on the homology of the homotopy
  projective limit are equal.
\end{proof}

\section{Comparison of analytic cyclic homology theories}
\label{sec:compare_HA_theories}

One of our main goals is to prove the following: if~\(D\) is a
dagger algebra such that the induced bornology on~\(D/\dvgen D\) is
the fine one, then \(\HA_*(D) \cong \HA_*(D/\dvgen D)\), where the
analytic cyclic homology \(\HA_*(D)\) of the dagger algebra~\(D\) is
defined in~\cite{Cortinas-Meyer-Mukherjee:NAHA}.  The proof shows
that \(\HAC(A)\) and \(\diss \HAC(D)\) are both quasi-isomorphic to
the complex \(\HAC(A,D,\varrho)\) defined in the previous section.
In this section, we state the relevant theorems and prove all but
one of them.  The longest proof will only come in the next section.

First, we prove that \(\HAC(A,D,\varrho)\) is isomorphic to
\(\diss \HAC(D)\) for suitable pro-dagger algebras~\(D\).  We also
construct a canonical quasi-isomorphism
\(\HAC(A) \to \HAC(A,W,\varrho)\) if~\(W\) is a projective system of
torsionfree complete bornological \(\dvr\)\nb-modules that is fine
mod~\(\dvgen\) and \(\varrho\colon W\onto A\) is an epimorphism,
that is, represented by a compatible family of surjective bounded
\(\dvr\)\nb-linear maps \(W_n \to A\).  We show already in this
section how to reduce this theorem to the special case where~\(W\)
is constant as a projective system.  This special case of the
theorem is proven only in Section~\ref{sec:independence_lifting}.
We also apply our comparison theorem to two examples that were
already computed in~\cite{Cortinas-Meyer-Mukherjee:NAHA}.

Now we compare \(\HAC(A,D,\varrho)\) for a pro-dagger algebra
\(D=(D_n)_{n\in\N}\) to the analytic cyclic homology chain complex
\(\HAC(D)\) defined in~\cite{Cortinas-Meyer-Mukherjee:NAHA}.  Define
\(\tens D = (\tens D_n)_{n\in\N}\) as above and let
\[
  \jens D \defeq \ker (\id_D^\#\colon \tens D \to D).
\]
The \(X\)\nb-complex for a complete bornological algebra~\(D\),
which is denoted by \(\comb{X}(D)\)
in~\cite{Cortinas-Meyer-Mukherjee:NAHA}, is built using completed
tensor products.  We denote it by~\(\comb{X}\) here to distinguish
it from the \(X\)\nb-complex with incomplete tensor products that we
use in Step~\ref{step:X} of the construction of
\(\HAC(A)\).  We simplify the definition of \(\HAC(D)\)
in~\cite{Cortinas-Meyer-Mukherjee:NAHA} using
\cite{Cortinas-Meyer-Mukherjee:NAHA}*{Remark~4.4.5}, which says that
the linear growth bornology on \(\tub{D}{I^l}\) relative to the
ideal \(\tub{I}{I^l}\) is equal to the usual linear growth bornology
because \(\tub{D}{I^l}\bigm/\tub{I}{I^l} \cong D\) is already a
dagger algebra.  This gives
\[
  \HAC(D) =
  \comb{X}\bigl(\comb{\ling{\tub{\tens D_n}{\jens D_n^l}}}
    \otimes \dvf\bigr)_{n\in\N, l\in\N^*}
\]
as a \(\Z/2\)\nb-graded chain complex of projective systems of
complete bornological \(\dvf\)\nb-vector spaces.  To compare it with
\(\HAC(A,D,\varrho)\), we study \(\diss \HAC(D)\) instead.

\begin{proposition}
  \label{pro:HAC_dagger_order}
  Let \(D=(D_n)_{n\in\N}\) be a pro-dagger algebra.  Then
  \[
    \diss \HAC(D) \cong
    \Bigl(\comb{\diss X(\ling{\tub{\tens D_n}{\jens D_n^l}}\otimes
      \dvf)}\Bigr)_{n\in\N, l\in\N^*}.
  \]
\end{proposition}

\begin{proof}
  The isomorphism is proven levelwise, that is, we may fix
  \(n\in\N\) and \(l\in\N^*\).  The bornological \(\dvr\)\nb-algebra
  \(\ling{\tub{\tens D_n}{\jens D_n^l}}\) is torsionfree by
  \cite{Meyer-Mukherjee:Bornological_tf}*{Proposition~4.11}.
  Therefore,
  \(\comb{\ling{\tub{\tens D_n}{\jens D_n^l}}} \otimes \dvf \cong
  \comb{\ling{\tub{\tens D_n}{\jens D_n^l}} \otimes \dvf}\) by
  \cite{Meyer-Mukherjee:Bornological_tf}*{Proposition~4.9}.  The
  complete and incomplete versions of the \(X\)\nb-complex are
  related by a canonical isomorphism
  \(\comb{X}(\comb{R}) \cong \comb{X(R)}\) for all bornological
  \(\dvr\)\nb-algebras~\(R\).  Therefore,
  \[
    \comb{X}\bigl(\comb{\ling{\tub{\tens D_n}{\jens D_n^l}}
      \otimes \dvf}\bigr)
    \cong \comb{X(\ling{\tub{\tens D_n}{\jens D_n^l}}\otimes \dvf)}.
  \]

  The description of the linear growth bornology on
  \(\tub{\tens D_n}{\jens D_n^\infty}\) in
  \cite{Cortinas-Meyer-Mukherjee:NAHA}*{Proposition~4.4.12} shows
  that \(X(\ling{\tub{\tens D_n}{\jens D_n^l}}\otimes \dvf)\)
  satisfies the third condition in Proposition~\ref{pro:diss_comb}.
  Namely, any subset of \(\tub{\tens D_n}{\jens D_n^\infty}\) of
  linear growth is contained in a bounded subset \(D_l(M,\alpha,f)\)
  that depends on a bounded \(\dvr\)\nb-submodule
  \(M\subseteq D_n\), \(\alpha\in \Q \cap (0,1/l)\) and
  \(f\in\N_0\).  The crucial feature is that the inclusion maps
  between the subsets \(D_l(M,\alpha,f)\) induce injective maps
  between their completions.  A closely related point is that the
  completed tensor product functor for complete torsionfree
  bornological \(\dvr\)\nb-modules preserves injectivity of maps
  (see \cite{Cortinas-Meyer-Mukherjee:NAHA}*{Proposition~2.4.5}).
  The same argument works for the odd part of the \(X\)\nb-complex,
  which is identified with a suitable space of noncommutative
  differential forms of odd degree, just as
  \(\tub{\tens D_n}{\jens D_n^\infty}\) is identified with a space
  of even-degree differential forms.  Now
  Proposition~\ref{pro:diss_comb} shows that we may change the order
  of dissection and completion.  This gives the desired isomorphism.
\end{proof}

A projective system of bornological algebras is called
\emph{analytically nilpotent} (see
\cite{Cortinas-Meyer-Mukherjee:NAHA}*{Definition~4.3.1}) if it is
nilpotent mod~\(\dvgen\) and isomorphic to a projective system of
dagger algebras.  A projective system of bornological algebras
\((B_n)_{n\in\N}\) is nilpotent mod~\(\dvgen\) if for any \(n\in\N\)
there are \(m\ge n\) and \(l\in\N^*\) so that the structure map
\(B_m \to B_n\) maps \(B_m^l\) into \(\dvgen\cdot B_n\).
In the following, we call morphisms of projective systems of
bornological algebras ``pro-homomorphisms''.

\begin{proposition}
  \label{Monsky-Washnitzer-general}
  Let~\(A\) be an \(\resf\)\nb-algebra.  Give~\(A\) the fine
  bornology and view it as a constant projective system of
  bornological \(\dvr\)\nb-algebras.  Let \(D = (D_n)_{n\in\N}\) be
  a projective system of dagger algebras and let
  \(\varrho \colon D \to A\) be a pro-homomorphism; it is
  represented by a compatible family of bounded algebra
  surjections \(\varrho_n \colon D_n \onto A\) for \(n\ge n_0\).
  Assume that \(\ker \varrho = (\ker \varrho_n)_{n\in\N}\) is
  analytically nilpotent.  Then
  \[
    \diss \HAC(D) \cong \HAC(A, D, \varrho).
  \]
\end{proposition}

\begin{proof}
  We use the description of \(\diss \HAC(D)\) in
  Proposition~\ref{pro:HAC_dagger_order}.  The only difference to
  the definition of \(\HAC(A,D,\varrho)\) is the tube algebra of
  \(\tens D\) that is used.  Namely, \(\diss \HAC(D)\) uses the tube
  algebra \(\tub{\tens D}{(\jens D)^\infty}\), whereas
  \(\HAC(A,D,\varrho)\) uses the tube algebra
  \(\tub{\tens D}{I^\infty}\) with
  \(I \defeq \ker (\varrho^\#\colon \tens D \to A)\).  The proof
  shows that these two tube algebras are
  isomorphic as projective systems of bornological algebras.

  Since \(\jens D \subseteq I\), the canonical pro-homomorphism
  \(\tens D \to \tub{\tens D}{I^\infty}\) factors uniquely through a
  pro-homomorphism
  \(\tub{\tens D}{(\jens D)^\infty} \to \tub{\tens D}{I^\infty}\).
  To construct its inverse, we form a pull back of extensions as in
  the following commuting diagram with exact rows:
  \[
    \begin{tikzcd}
      \tub{\jens D}{(\jens D)^\infty} \ar[r, dotted, >->] \ar[d, equal] &
      E \ar[dr,  phantom, "\ulcorner", very near start]
      \ar[r, dotted, ->>]  \ar[d, dotted, >->] &
      \ker \varrho \ar[d, >->]\\
      \tub{\jens D}{(\jens D)^\infty} \ar[r, >->] &
      \tub{\tens D}{(\jens D)^\infty} \ar[r, ->>] \ar[d] &
      D  \ar[d]\\
      & A \ar[r, equal] & A
    \end{tikzcd}
  \]
  By assumption, \(\ker \varrho\) is nilpotent mod~\(\dvgen\).  So
  is \(\tub{\jens D}{(\jens D)^\infty}\) by
  \cite{Cortinas-Meyer-Mukherjee:NAHA}*{Proposition~4.2.4}.
  Then~\(E\) is nilpotent mod~\(\dvgen\) by
  \cite{Cortinas-Meyer-Mukherjee:NAHA}*{Proposition~4.2.6}.  The
  universal property of pullbacks implies that~\(E\) is the kernel
  of the vertical map \(\tub{\tens D}{(\jens D)^\infty} \to A\).  So
  the canonical pro-homomorphism
  \(\tens D \to\tub{\tens D}{(\jens D)^\infty}\) maps~\(I\)
  to~\(E\).  Since~\(E\) is analytically nilpotent,
  \cite{Cortinas-Meyer-Mukherjee:NAHA}*{Proposition~4.2.4} gives a
  pro-homomorphism
  \(\tub{\tens D}{I^\infty} \to \tub{\tens D}{(\jens D)^\infty}\).
  The composite maps between \(\tub{\tens D}{I^\infty}\) and
  \(\tub{\tens D}{(\jens D)^\infty}\) in both directions extend the
  canonical maps \(\tens D \to \tub{\tens D}{I^\infty}\) and
  \(\tens D \to \tub{\tens D}{(\jens D)^\infty}\), respectively.
  Then the uniqueness part of the universal property of tube
  algebras in
  \cite{Cortinas-Meyer-Mukherjee:NAHA}*{Proposition~4.2.4} shows
  that the pro-homomorphisms
  \(\tub{\tens D}{I^\infty} \leftrightarrow \tub{\tens D}{(\jens
    D)^\infty}\) built above are inverse to each other.
\end{proof}

\begin{corollary}
  \label{cor:Monsky-Washnitzer_lift}
  Let~\(D\) be a dagger algebra.  Let \(A\defeq D/\dvgen D\) and let
  \(\varrho\colon D\onto A\) be the quotient map.  Then
  \(\HAC(D/\dvgen D, D , \varrho)\cong \diss \HAC(D)\).
\end{corollary}

\begin{proof}
  This is the special case of
  Proposition~\ref{Monsky-Washnitzer-general} where~\(D\) is a
  constant projective system and \(\ker \varrho = \dvgen D\).  This
  constant projective system is nilpotent mod~\(\dvgen\) because
  \((\dvgen D)^2 \subseteq \dvgen\cdot \dvgen D\).
\end{proof}

Now we compare \(\HAC(A,W,\varrho)\) to \(\HAC(A)\).  We first
assume~\(W\) to be a constant projective system.

\begin{lemma}
  \label{lem:compare_HAC_A_to_W}
  Let~\(A\) be an \(\resf\)\nb-algebra.  Let~\(W\) be a complete,
  torsionfree bornological \(\dvr\)\nb-module, and let
  \(\varrho\colon W\to A\) be a \(\dvr\)\nb-module map.
  If~\(\varrho\) is surjective, then there is a map
  \(s\colon A \to W\) with \(\varrho\circ s = \id_A\).  This induces
  a bounded \(\dvr\)\nb-linear map \(s^\#\colon \dvr[A] \to W\) with
  \(\varrho\circ s^\# = \id_A^\#\), which further induces a chain
  map
  \[
    (\id_A,s^\#)_*\colon \HAC(A) \to \HAC(A,W,\varrho).
  \]
\end{lemma}

\begin{proof}
  The map~\(s\) exists because~\(\varrho\) is surjective.  It
  induces a \(\dvr\)\nb-linear map \(s^\#\colon \dvr[A] \to W\),
  which is bounded because~\(\dvr[A]\) carries the fine bornology.
  The equation \(\varrho\circ s^\# = \id_A^\#\) follows from
  \(\varrho\circ s = \id_A\).  The pair of maps
  \(\id_A\colon A\to A\) and \(s^\#\colon \dvr[A] \to W\) induces
  the desired chain map \((\id_A,s^\#)_*\) by
  Lemma~\ref{lem:hybrid_functorial}.
\end{proof}

The following theorem is one of our main results.  Its proof is long
and will occupy all of Section~\ref{sec:independence_lifting}.
Before we prove it, we deduce some further results from it.

\begin{theorem}
  \label{the:HA_through_lift}
  Let~\(A\) be an \(\resf\)\nb-algebra.  Let~\(W\) be a complete,
  torsionfree bornological \(\dvr\)\nb-module, and let
  \(\varrho\colon W\to A\) be a \(\dvr\)\nb-module map.  Assume
  that~\(\varrho\) is surjective and that the induced bornology on
  \(W/\dvgen W\) is the fine one.  Then the chain map
  \(\HAC(A) \to \HAC(A,W,\varrho)\) in
  Lemma~\textup{\ref{lem:compare_HAC_A_to_W}} is a
  quasi-isomorphism.  Thus the induced map
  \(\HA_*(A) \to \HA_*(A,W,\varrho)\) is an isomorphism of
  \(\dvf\)\nb-vector spaces.
\end{theorem}

\begin{corollary}
  Let~\(A\) be the coordinate ring of a smooth, affine variety~\(X\)
  over~\(\resf\) of relative dimension~\(1\).  Let~\(R\) be a
  smooth, commutative \(\dvr\)\nb-algebra with
  \(R/\dvgen R \cong A\).  Give~\(R\) the fine bornology.  For
  \(*=0,1\), the following homology groups are isomorphic:
  \begin{enumerate}
  \item \(\HA_*(A)\);
  \item \(\HA_*(A, R^\updagger, R^\updagger \onto A)\);
  \item \(\HA_*(R^\updagger)\);
  \item the de Rham cohomology of \(R^\updagger \otimes \dvf\);
  \item the Monsky--Washnitzer cohomology of~\(X\).
  \end{enumerate}
  This is isomorphic to the rigid cohomology
  \(H_{\mathrm{rig}}^*(A,\dvf)\) of~\(X\) if
  \(\mathrm{char}(\resf)\neq 0\).
\end{corollary}

\begin{proof}
  The first two homology groups are isomorphic by
  Theorem~\ref{the:HA_through_lift}, the next two by
  Corollary~\ref{cor:Monsky-Washnitzer_lift}.  The remaining ones are
  isomorphic by
  \cite{Cortinas-Meyer-Mukherjee:NAHA}*{Theorem~9.2.9}.
\end{proof}

\begin{definition}
  A projective system of bornological \(\dvr\)\nb-modules
  \(W=(W_n)_{n\in\N}\) is called \emph{fine mod~\(\dvgen\)} if the
  quotient bornology on \(W_n/\dvgen W_n\) is the fine one for all
  \(n\in\N\).  We also apply this to bornological
  \(\dvr\)\nb-modules, viewed as constant projective systems.
\end{definition}

\begin{theorem}
  \label{the:HA_through_pro-lift}
  Let~\(A\) be an \(\resf\)\nb-algebra, viewed as a constant
  projective system of bornological \(\dvr\)\nb-algebras with the fine
  bornology.  Let~\(W\) be a
  projective system of torsionfree complete bornological
  \(\dvr\)\nb-modules that is fine mod~\(\dvgen\).
  Let \(\varrho \colon W \onto A\) be an epimorphism in the category
  of projective systems of bornological \(\dvr\)\nb-modules.  There
  is a quasi-isomorphism
  \(\HAC(A) \simeq \HAC(A, W, \varrho)\).
\end{theorem}

\begin{proof}
  Write \(W\) as a projective system with entries \(W_n\) and maps
  \(\varphi_n^{n+1}\colon W_{n+1} \to W_n\) for \(n\in \N\).  The
  morphism~\(\varrho\) is represented by a bounded
  \(\dvr\)\nb-module homomorphism
  \(\varrho_{n_0}\colon W_{n_0} \to A\) for some \(n_0\in\N\).
  Reindexing, we may arrange without loss of generality that
  \(n_0=0\).  Let
  \(\varrho_n \defeq \varrho_0 \circ \varphi_0^n\colon W_n \to W_0
  \to A\).  Each map~\(\varrho_n\) is surjective because~\(\varrho\)
  is an epimorphism.  For fixed \(n\in\N\), there is a section
  \(s_n \colon A \to W_n\) for \(\varrho_n \colon W_n \onto A\).  By
  Theorem~\ref{the:HA_through_lift}, \(s_n\) induces a
  quasi-isomorphism \(\HAC(A) \simeq \HAC(A, W_n, \varrho_n)\).  By
  Proposition~\ref{pro:homotopy_lifting_2}, the linear homotopy
  between \(\varphi_n^{n+1}\circ s_{n+1}\) and~\(s_n\) induces a
  chain homotopy between the quasi-isomorphisms
  \(\HAC(A) \rightrightarrows \HAC(A, W_n, \varrho_n)\) induced by
  \(\varphi_n^{n+1}\circ s_{n+1}\) and~\(s_n\).  As a result, the
  maps
  \[
    (\id_A,\varphi_n^{n+1})_*\colon
    \HAC(A,W_{n+1},\varrho_{n+1}) \to \HAC(A,W_n,\varrho_n)
  \]
  are quasi-isomorphisms for all~\(n\).  Since \(\HAC(A,W,\varrho)\)
  is the formal projective limit of the complexes
  \(\HAC(A,W_n,\varrho_n)\), the canonical map
  \(\HAC(A,W,\varrho) \to \HAC(A,W_0,\varrho_0)\) is a
  quasi-isomorphism.  Therefore,
  \(\HAC(A,W,\varrho) \simeq \HAC(A)\).
\end{proof}

\begin{theorem}
  \label{the:pro-dagger_lifting}
  Let~\(A\) be an \(\resf\)\nb-algebra.  Let \(D=(D_n)_{n\in\N}\) be
  a projective system of dagger algebras that is fine mod~\(\dvgen\), and let
  \(\varrho\colon D\to A\) be a pro-homomorphism, represented by a
  coherent family of surjective homomorphisms \(\varrho_n\colon D_n \to A\) for
  \(n\in\N\).  Assume that
  \(\ker \varrho = (\ker \varrho_n)_{n\in\N}\) is analytically
  nilpotent.  Then there is a canonical quasi-isomorphism
  \(\HAC(A) \simeq \diss \HAC(D)\).
\end{theorem}

\begin{proof}
  Theorem~\ref{the:HA_through_pro-lift} gives a canonical
  quasi-isomorphism \(\HAC(A) \simeq \HAC(A, D, \varrho)\).
  Proposition~\ref{Monsky-Washnitzer-general} identifies this
  further with \(\diss \HAC(D)\).
\end{proof}

\begin{example}
  \label{exa:Leavitt_path_algebras}
  Let~\(E\) be a directed graph with countably many vertices.  The
  Leavitt path algebra \(L(R,E)\) and the Cohn path algebra
  \(C(R,E)\) of~\(E\) with coefficients in a commutative ground
  ring~\(R\) are defined in
  \cite{Abrams-Ara-Siles-Molina:Leavitt_path}*{Definitions 1.2.3 and
    1.2.5}.  The chain complexes \(\HAC(C(\dvr,E)^\updagger)\) and
  \(\HAC(L(\dvr,E)^\updagger)\) are computed in
  \cite{Cortinas-Meyer-Mukherjee:NAHA}*{Theorem~8.1} (up to a typo,
  replace~\(\dvr\) by~\(\dvf\)).  Namely, let~\(\dvf^{(E^0)}\)
  denote the free \(\dvf\)\nb-vector space on the set of
  vertices~\(E^0\) and let~\(N_E\) denote the \(\dvf\)\nb-linear map
  \(\dvf^{\mathrm{reg}(E)} \to \dvf^{(E^0)}\) induced by the matrix
  with entries \(\delta_{v,w} - \abs{s^{-1}(w) \cap r^{-1}(v)}\) for
  \(v\in E^0\), \(w\in \mathrm{reg}(E)\), where \(\mathrm{reg}(E)\)
  is the set of regular vertices.  Then
  \[
    \HAC(C(\dvr,E)^\updagger) \cong \dvf^{(E^0)},\qquad
    \HAC(L(\dvr,E)^\updagger) \cong \coker (N_E) \oplus \ker(N_E)[1].
  \]
  It is easy to see that the dagger completed Leavitt and Cohn path
  algebras over~\(\dvr\) are fine mod~\(\dvgen\), and their
  reductions mod~\(\dvgen\) are the Leavitt and Cohn path algebras
  over~\(\resf\), respectively.  So
  Theorem~\ref{the:pro-dagger_lifting} implies quasi-isomorphisms
  \[
    \HAC(C(\dvr,E)^\updagger) \simeq \HAC(C(\resf,E)),\qquad
    \HAC(L(\dvr,E)^\updagger) \simeq \HAC(L(\resf,E)).
  \]
\end{example}

\section{Independence of the choice of lifting}
\label{sec:independence_lifting}

In this section, we prove Theorem~\ref{the:HA_through_lift}.  First,
we recall a theorem about the local structure of torsionfree
complete bornological \(\dvr\)\nb-modules.

\begin{definition}
  \label{def:Cont0}
  Let~\(B\) be a set.  Let \(\Cont_0(B,\dvr)\) be the set of all
  functions \(f\colon B \to \dvr\) such that for each \(\delta>0\)
  there is a finite subset \(F\subseteq B\) with
  \(\abs{f(x)}<\delta\) for all \(x\in B\setminus F\).  Define
  \(\Cont_0(B,\dvf)\) similarly.  Equip both with the supremum norm.
  This makes \(\Cont_0(B, \dvf)\) a Banach \(\dvf\)\nb-vector space
  and \(\Cont_0(B, \dvr)\) its unit ball.
\end{definition}

\begin{theorem}[\cite{Cortinas-Meyer-Mukherjee:NAHA}*{Theorem~2.4.2}]
  \label{the:structure}
  Let~\(W\) be a torsionfree complete bornological
  \(\dvr\)\nb-\hspace{0pt}module.  Any \(\dvgen\)\nb-adically
  complete bounded \(\dvr\)\nb-submodule~\(M\) of~\(W\) is
  isomorphic to \(\Cont_0(B,\dvr)\) for some set~\(B\).
\end{theorem}

Since \(\Cont_0(B, \dvr)\) has a nice basis, we may define bounded
\(\dvr\)\nb-linear maps on it rather easily.  We will feed these
maps into the functoriality and homotopy invariance results proven
in Section~\ref{sec:homotopy_stability}.  This uses the
following proposition:

\begin{proposition}
  \label{pro:HAC_localisation}
  Let~\(A\) be an \(\resf\)\nb-algebra, let~\(W\) be a torsionfree
  complete bornological \(\dvr\)\nb-module and let
  \(\varrho\colon W\to A\) be a \(\dvr\)\nb-linear map.  Let
  \(I\defeq \ker (\varrho^\#\colon \tens W \to A)\) and let
  \(l\in\N^*\).  Assume that~\(W\) is fine mod~\(\dvgen\).  Then
  \(\tub{\tens W}{I^l}\) is the strict inductive
  limit of \(\tub{\tens M}{(\tens M\cap I)^l}\) for the
  \(\dvgen\)\nb-adically complete bounded \(\dvr\)\nb-submodules
  \(M\subseteq W\); strict means that the structure maps
  \[
    \tub{\tens M}{(\tens M\cap I)^l} \to
    \tub{\tens M'}{(\tens M'\cap I)^l}
  \]
  for \(M\subseteq M'\) are all injective.
\end{proposition}

Recall that \(\HAC(A,W,\varrho)\) is the projective
system of complexes
\begin{equation}
  \label{eq:HAC_l}
  \HAC(A,W,\varrho,l)
  \defeq \comb{\diss X( \ling{\tub{\tens W}{I^l}} \otimes \dvf)}
\end{equation}
for \(l \in \N^*\).  The proposition implies that
\(\HAC(A,W,\varrho,l) = \varinjlim \HAC(A,M,\varrho|_M,l)\)
for~\(M\) running through the directed set of \(\dvgen\)\nb-adically
complete bounded \(\dvr\)\nb-submodules \(M\subseteq W\); this is
because taking linear growth bornologies, tensoring with~\(\dvf\),
taking the \(X\)\nb-complex, dissection and completion commute with
inductive limits, being ``local'' constructions.  Since inductive
limits in the category of projective systems are not levelwise,
\(\HAC(A,W,\varrho)\) itself is \emph{not} an inductive limit of
\(\HAC(A,M,\varrho|_M)\).

The proof of the proposition uses a few lemmas.

\begin{lemma}
  \label{lem:fine_quotients}
  Let~\(W\) be a torsionfree bornological
  \(\dvr\)\nb-module.  If the quotient bornology on
  \(W/\dvgen W\) is fine, then the same is true for
  \(W/\dvgen^m W\) for all \(m\in\N\).
\end{lemma}

\begin{proof}
  The proof is by induction on~\(m\).  The case \(m=0\) is trivial.
  Assume that the claim is true for~\(m\).  Let \(S\subseteq W\) be
  a bounded \(\dvr\)\nb-module.  We must show that its image in
  \(W/\dvgen^{m+1} W\) is finitely generated.  By assumption, its
  image in \(W/\dvgen W\) is finitely generated.  So we may pick a
  finite set \(x_1,\dotsc,x_n\) with
  \(S \subseteq \sum_{j=1}^n \dvr x_j + \dvgen W\).  Let
  \(S_1 \defeq \Bigl(S + \sum_{j=1}^n \dvr x_j\Bigr) \cap \dvgen
  W\).  Then \(S_1\subseteq \dvgen W\) is bounded and
  \(S\subseteq \sum_{j=1}^n \dvr x_j + S_1\).  Since~\(W\) is
  bornologically torsionfree, \(\dvgen^{-1} S_1\) is bounded as
  well.  And
  \(S \subseteq \sum_{j=1}^n \dvr x_j + \dvgen\cdot
  (\dvgen^{-1}S_1)\).  The induction assumption applied to the image
  of~\(\dvgen^{-1} S_1\) in~\(W/\dvgen^m W\) gives finitely many
  elements \(x_{n+1},\dotsc,x_k\) such that
  \(\dvgen^{-1} S_1 \subseteq \sum_{j=n+1}^k \dvr x_j + \dvgen^m
  W\).  Then \(S \subseteq \sum_{j=1}^k \dvr x_j + \dvgen^{m+1} W\)
  as desired.
\end{proof}

\begin{lemma}
  \label{lem:fine_quotient_tensor}
  Let \(W_1,W_2\) be torsionfree bornological \(\dvr\)\nb-modules
  that are fine mod~\(\dvgen\).  Then the same is true for
  \(W_1 \otimes W_2\) and \(\tens W_1\).
\end{lemma}

\begin{proof}
  Any bounded subset of \(W_1 \otimes W_2\) is contained in the
  image of \(M_1 \otimes M_2\) for bounded \(\dvr\)\nb-submodules
  \(M_j \subseteq W_j\) for \(j=1,2\).  By assumption, for \(j=1,2\)
  there are finite subsets \(S_j \subseteq W_j\) with
  \(M_j \subseteq \sum_{x\in S_j} \dvr\cdot x + \dvgen W_j\).  Hence the
  image of \(M_1 \otimes M_2\) in \(W_1 \otimes W_2\) is contained
  in
  \(\sum_{x\in S_1,y\in S_2} \dvr\cdot (x\otimes y) + \dvgen\cdot
  W_1 \otimes W_2\).  Then \(W_1 \otimes W_2\) is fine
  mod~\(\dvgen\).  By induction, it follows that \(W_1^{\otimes n}\)
  is fine mod~\(\dvgen\) for all \(n\in\N\).  This is further
  inherited by the direct sum
  \(\tens W_1 = \bigoplus W_1^{\otimes n}\) because any bounded
  subset of \(\tens W_1\) is already contained in a finite subsum.
\end{proof}

\begin{lemma}
  \label{lem:bounded_in_tensor_tube_technical1}
  In the situation of
  Proposition~\textup{\ref{pro:HAC_localisation}}, let~\(N\) be a
  bounded \(\dvr\)\nb-submodule of \(\tub{\tens W}{I^l}\).  There
  are \(a\in \N\), a finite subset \(S\subseteq I\), and a bounded
  subset \(B\subseteq \tens W\) such that~\(N\) is contained in
  \(\sum_{j=1}^a \dvr \dvgen^{-j} S^{l j} + B\).
\end{lemma}

\begin{proof}
  By construction, \(\tub{\tens W}{I^l}\) carries the subspace
  bornology from \(\tens W \otimes \dvf\).  Hence~\(N\) is bounded
  in \(\tens W \otimes \dvf\).  Then it is contained in
  \(\dvgen^{-b} N'\) for a bounded subset \(N'\subseteq \tens W\).
  Lemmas \ref{lem:fine_quotients} and~\ref{lem:fine_quotient_tensor}
  imply that the quotient bornology on \(\tens W/\dvgen^b \tens W\)
  is fine.  Since~\(\dvr\) is Noetherean, there are
  \(x_1,\dotsc,x_n\in N'\subseteq \tub{\tens W}{I^l} =
  \sum_{j=0}^\infty \dvgen^{-j} I^{j l}\) that generate the image
  of~\(N'\) in \(\tens W/\dvgen^b \tens W\) as a \(\dvr\)\nb-module.
  Each \(x_j\in N'\) may be written as a polynomial in
  \(\dvgen^{-1} y_1 \dotsm y_l\) with \(y_1,\dotsc,y_l \in I\), plus
  a term in~\(\tens W\).  Let \(S \subseteq I\) be the (finite) set
  of all factors~\(y_m\) that appear in these products.  Then each
  element of~\(N\) may be written as a finite \(\dvr\)\nb-linear
  combination of elements of \(\bigcup_{j=1}^a \dvgen^{-j} S^{j l}\)
  for some \(a \in \N\), plus an element of~\(\tens W\).  Let
  \(B \defeq \Bigl(N + \sum_{j=1}^a \dvr \dvgen^{-j} S^{j l}\Bigr)
  \cap \tens W\).  This is bounded because~\(\tens W\) is
  bornologically torsionfree.  By construction, \(N\) is contained
  in \(\sum_{j=1}^a \dvr \dvgen^{-j} S^{l j} + B\).
\end{proof}

\begin{proof}[Proof of Proposition~\textup{\ref{pro:HAC_localisation}}]
  Let \(M\subseteq W\) be a bounded \(\dvr\)\nb-submodule.  The
  canonical map \(\tens M \to \tens W\) is injective.  This remains
  so after tensoring with~\(\dvf\) because \(\tens W\) is
  torsionfree.  Then the canonical map
  \(\tub{\tens M}{(\tens M\cap I)^l} \to \tub{\tens W}{I^l}\) is
  injective because both tube algebras are defined as subalgebras of
  \(\tens M \otimes \dvf\) and \(\tens W \otimes \dvf\),
  respectively.  Therefore, the inductive system formed by
  \(\tub{\tens M}{(\tens M\cap I)^l}\) for the
  \(\dvgen\)\nb-adically complete bounded \(\dvr\)\nb-submodules
  \(M\subseteq W\) is strict and the induced bounded map
  \(\varinjlim \tub{\tens M}{(\tens M\cap I)^l} \to \tub{\tens
    W}{I^l}\) is injective.  To show that it is a bornological
  isomorphism, we must prove that any bounded subset
  \(N\subseteq \tub{\tens W}{I^l}\) is the image of a bounded subset
  in \(\tub{\tens M}{(\tens M\cap I)^l}\) for some
  \(\dvgen\)\nb-adically complete bounded \(\dvr\)\nb-submodule
  \(M\subseteq W\).

  Let \(B\) and~\(S\) be as in
  Lemma~\ref{lem:bounded_in_tensor_tube_technical1}.  The subset
  \(B\cup S\) of \(\tens W\) is bounded and hence contained in the
  image of \(\bigoplus_{j=1}^b M^{\otimes j}\) for some
  \(\dvgen\)\nb-adically complete, bounded \(\dvr\)\nb-submodule
  \(M\subseteq W\) and some \(b\in\N\).  By construction of \(B\)
  and~\(S\), the subset~\(N\) is the image of a bounded subset of
  \(\tub{\tens M}{(I\cap \tens M)^l}\).
\end{proof}

As in Theorem~\ref{the:HA_through_lift}, let~\(A\) be an
\(\resf\)\nb-algebra, let~\(W\) be a torsionfree complete
bornological \(\dvr\)\nb-module and let \(\varrho\colon W\to A\) be
a \(\dvr\)\nb-linear map.  We assume that~\(W\) is fine
mod~\(\dvgen\) and that~\(\varrho\) is surjective.  Then there is a
section \(s \colon A \to W\) of~\(\varrho\).  It induces a bounded
\(\dvr\)\nb-linear map \(s^\# \colon \dvr[A] \to W\) that lifts the
identity map on~\(A\).  By Lemma~\ref{lem:hybrid_functorial}, the
pair of maps \((s^\#,\id_A)\) induces a chain map
\[
  (s^\#,\id_A)\colon
  \HAC(A) \defeq \HAC(A,\dvr[A],\id_A^\#) \to \HAC(A,W,\varrho).
\]
We must show that \((s^\#,\id_A)\) is a quasi-isomorphism.

If the identity map on~\(A\) would lift to a bounded
\(\dvr\)\nb-linear map \(f\colon W\to\dvr[A]\), that would give us a
chain map \((f,\id_A)_*\) in the opposite direction.  Since any two
liftings of the same map on~\(A\) are homotopic by a linear
homotopy, it would then follow from
Proposition~\ref{pro:homotopy_lifting_2} that \((s^\#,\id_A)_*\) and
\((f,\id_A)_*\) are inverse to each other up to chain homotopy.  In
particular, it would follow that \((s^\#,\id_A)_*\) is a chain
homotopy equivalence.  This argument is analogous to the proof of
\cite{Cortinas-Meyer-Mukherjee:NAHA}*{Corollary~4.3.12}.
Unfortunately, such a map~\(f\) need not exist.  We only know the
following weaker result:

\begin{lemma}
  \label{lem:local_inverses_W_to_VA}
  Let \(M\subseteq W\) be a \(\dvgen\)\nb-adically complete bounded
  \(\dvr\)\nb-submodule.  There is a bounded \(\dvr\)\nb-linear map
  \(f_M\colon M\to \dvr[A]\) with
  \(\id_A^\# \circ f_M = \varrho|_M\colon M\to A\).
\end{lemma}

\begin{proof}
  The image \(\varrho(M) \subseteq A\) has finite dimension
  because~\(W\) is fine mod~\(\dvgen\), so that the image of~\(M\)
  in \(W/\dvgen W\) has finite dimension.  Let
  \(X\defeq \{a_1,\dotsc,a_j\} \subseteq A\) be a basis for
  \(\varrho(M)\).  Identify~\(M\) with \(\Cont_0(B,\dvr)\) by
  Theorem~\ref{the:structure}.  For \(i\in B\), let
  \(\delta_i \in \Cont_0(B,\dvr)\) be the characteristic function
  of~\(\{i\}\).  Write \(\varrho(\delta_i)\) as a linear combination
  of~\(X\) with coefficients in~\(\resf\) and lift these
  coefficients to~\(\dvr\) in any way.  This defines an element of
  \(\dvr [X] \subseteq \dvr[A]\), which we call \(f_M(\delta_i)\).
  Since~\(\dvr[X]\) is \(\dvgen\)\nb-adically complete, there is a
  unique bounded \(\dvr\)\nb-linear map
  \(f_M\colon M \cong \Cont_0(B,\dvr) \to \dvr [X]\) with the
  specified values on all characteristic functions.  Its
  image~\(\dvr[X]\) is bounded in the fine bornology
  on~\(\dvr [A]\).  By construction,
  \(\id_A^\# \circ f_M(\delta_i) = \varrho(\delta_i)\) for all
  \(i\in B\).  This implies \(\id_A^\# \circ f_M = \varrho|_M\).
\end{proof}

By Lemma~\ref{lem:hybrid_functorial}, the pair of maps
\((f_M,\id_A)\) induces a chain map
\[
  (f_M,\id_A)_*\colon
  \HAC(A,M,\varrho|_M) \to \HAC(A,\dvr[A],\id_A^\#) = \HAC(A).
\]
We also need to know that \((f_M,\id_A)_*\) is inverse to
\((s^\#,\id_A)\) up to chain homotopy in some sense.  First,
consider the composite map
\[
  (s^\#,\id_A)\circ (f_M,\id_A)_*
  = (s^\#\circ f_M,\id_A)_*
  \colon \HAC(A,M,\varrho|_M) \to \HAC(A,W,\varrho).
\]
Both \(s^\#\circ f_M\) and the inclusion
\(i_M\colon M\hookrightarrow W\) lift the identity map on~\(A\) to
bounded \(\dvr\)\nb-linear maps \(M\to W\).  The linear homotopy
\[
  H_M \defeq (1-t) i_M + t\cdot (s^\#\circ f_M)\colon
  M \to W\otimes \dvr[t]
\]
lifts the constant homotopy \(c\colon A\to A\otimes \dvr[t]\),
\(a\mapsto a\cdot 1\), on the identity map on~\(A\).  By
Proposition~\ref{pro:homotopy_lifting_2}, \((H_M,c)\) induces a
chain homotopy between \((s^\#,\id_A)\circ (f_M,\id_A)_*\) and the
canonical chain map \(\HAC(A,M,\varrho|_M) \to \HAC(A,W,\varrho)\).
To define a composition in the reverse order, we pick a finite
subset \(X\subseteq A\) with \(s(X) \subseteq M\).  Then we get
an induced bounded \(\dvr\)\nb-linear map
\(s_X^\# = s^\#|_{\dvr[X]}\colon \dvr[X] \to M\), which induces a chain
map
\[
  (s_X^\#,\id_A)_*\colon
  \HAC(A,\dvr[X],\id_A^\#|_{\dvr[X]}) \to \HAC(A,M,\varrho|_M).
\]
Let \(i_X\colon \dvr[X] \hookrightarrow \dvr[A]\) be the inclusion
map.  The constant homotopy on~\(A\) lifts to a linear homotopy
\[
  H_X \defeq
  (1-t)\cdot i_X + t\cdot (f_M \circ s_X^\#)\colon
  \dvr[X] \to \dvr[A].
\]
The pair \((H_X,c)\) induces a chain homotopy between
\((f_M,\id_A)_*\circ (s_X^\#,\id_A)_*\) and the canonical chain map
\(\HAC(A,X,\id_A^\#|_{\dvr[X]}) \to \HAC(A)\) by
Proposition~\ref{pro:homotopy_lifting_2}.

Next, we introduce a concept of ``local chain homotopy equivalence''
with two properties.  First, any local chain homotopy equivalence is
a quasi-isomorphism.  Secondly,
\((s^\#,\id_A)_*\colon \HAC(A) \to \HAC(A,W,\varrho)\) is a local
chain homotopy equivalence.

\begin{definition}
  \label{def:locally_contractible}
  Let \(C=(C_k,\alpha_n^k)_{k\in\N}\) be a chain complex over
  \(\overleftarrow{\mathsf{Ind}(\mathsf{Ban}_\dvf)}\).  We may arrange for
  each~\(C_k\) to be a chain complex and write
  \(C_k \cong (C_{k,i})_{i\in I_k}\) as an inductive system of chain
  complexes.  For each \(n,k\in\N\), \(i\in I_k\), with \(k\ge n\),
  let \(\alpha_{n,i}^k \colon C_{k,i} \to C_n\) be the component of
  the structure map \(\alpha_n^k\colon C_k \to C_n\) of the
  projective system at~\(i\); this is a bounded map
  to~\(C_{n,j}\) for some \(j\in I_n\).  The chain complex~\(C\) is
  called \emph{locally contractible} if, for every~\(n\), there is a
  \(k \geq n\) such that for any \(i \in I_k\), the
  map~\(\alpha_{n,i}^k\) is null-homotopic.  A chain map
  \(f \colon C \to D\) is called a \emph{local chain homotopy
    equivalence} if its mapping cone is locally contractible.
\end{definition}

\begin{lemma}
  \label{lem:locally_contractible_exact}
  Locally contractible chain complexes are exact.
\end{lemma}

\begin{proof}
  Let~\(C\) be a locally contractible chain complex.  Write
  \(C \cong (C_k,d_k)_{k \in \N}\) with a compatible family of
  morphisms \(d_k \colon C_k \to C_k\) in
  \(\mathsf{Ind}(\mathsf{Ban}_\dvf)\) with \(d_k^2=0\), as in the
  definition of a locally contractible chain complex.  Then
  \(\ker(d) \cong \ker(d_n)_{n\in \N}\).  We need to prove that the
  morphism of projective systems described by the morphisms
  \(d_n \colon C_n \to \ker(d_n)\) is a cokernel in the
  category \(\overleftarrow{\mathsf{Ind}(\mathsf{Ban}_\dvf)}\).
  Let \((C_{k,i},d_{k,i})\), \(\alpha_n^k\) and \(\alpha_{n,i}^k\) be as in
  Definition~\ref{def:locally_contractible}.  Let \(n\in\N\).
  Since~\(C\) is locally contractible, there is \(k \geq n\) such
  that for each \(i \in I_k\), there is a map
  \(h_{n,i}^k \colon C_{k,i} \to C_n\) with
  \[
    h_{n,i}^k \circ d_{k,i} + d_n \circ h_{n,i}^k = \alpha_{n,i}^k.
  \]
  We replace~\(h_{n,i}^k\) by its restriction to \(\ker d_k\), which
  satisfies \(d_n \circ h_{n,i}^k = \alpha_{n,i}^k\).  Composing
  with the structure maps~\(\alpha_k^l\), we get such maps for all
  \(l\ge k\) and \(i\in I_l\) as well.  For \(l\ge k\), we build a
  pull-back diagram
  \[
    \begin{tikzcd}
      X_{l,n} \ar[r, "g_{l,n}"] \ar[d, "\gamma_{l,n}"']
      \arrow[dr,phantom, "\ulcorner", very near start] &
      \ker(d_l) \ar[d, "\alpha_n^l"] \\
      C_n \ar[r, "d_n"'] &
      \ker(d_n).
    \end{tikzcd}
  \]
  The universal property of pullbacks gives a unique map
  \(\sigma_{l,i}^n \colon \ker(d_l)_i \to X_{l,n}\) with
  \(g_{l,n}\circ \sigma_{l,n}^i = \mathrm{can}_i\colon \ker(d_l)_i
  \to \ker(d_l)\) and
  \(\gamma_{l,n} \circ \sigma_{l,n}^i = h_{n,i}^l\).  Then
  \(g_{l,n} \colon X_{l,n} \to \ker(d_l)\) is a cokernel in
  \(\mathsf{Ind}(\mathsf{Ban}_\dvf)\).  The maps~\((g_{l,n})\)
  combine to a morphism of pro-ind systems.  This morphism is a
  cokernel because each~\(g_{l,n}\) is a cokernel.  Since the family
  of maps \(\ker (d_l) \to \ker (d_n)\) represents the identity map
  of projective systems, \(X\) is isomorphic as a projective system
  to~\(C\), and the maps~\((g_{l,n})\) represent the map
  \(d\colon C\to \ker(d)\).  Consequently,
  \(d \colon C \to \ker(d)\) is a cokernel.
\end{proof}

The following lemma describes local chain homotopy equivalences
directly without referring to the mapping cone:

\begin{lemma}
  \label{lem:explicit_local_homotopy_equivalences}
  Let \(f \colon C \to D\) be a chain map in
  \(\overleftarrow{\mathsf{Ind}(\mathsf{Ban}_\dvf)}\).  We may
  represent~\(f\) by a compatible family
  \((f_n \colon C_n \to D_n)_{n \in \N}\) of chain maps in
  \(\mathsf{Ind}(\mathsf{Ban}_\dvf)\), and each~\(f_n\) by a coherent
  family of chain maps \(f_{n,i}\colon C_{n,i} \to D_{n,i}\)
  in~\(\mathsf{Ban}_\dvf\) for \(i\in I_n\) with some directed
  set~\(I_n\).  Suppose that for each \(n \in \N\), there is an
  \(m \geq n\), such that for each \(i \in I_m\), there are
  morphisms
  \[
    g_{m,i}^n \colon D_{m,i} \to C_n,\qquad
    h_{m,i}^D \colon D_{m,i} \to D_n[1],\qquad
    h_{m,i}^C \colon C_{m,i} \to C_n[1],
  \]
  where~\(g_{m,i}^n\) are chain maps and \(h_{m,i}^D\) and
  \(h_{m,j}^C\) are chain homotopies between \(f_n \circ g_{m,i}^n\)
  and \(g_{m,i} \circ f_{m,i}\), and the canonical maps
  \(\eta_{m,i}^n \colon D_{m,i} \to D_n\) and
  \(\gamma_{m,i}^n \colon C_{m,i} \to C_n\), respectively.
  Then~\(f\) is a local chain homotopy equivalence.
\end{lemma}

\begin{proof}
  We need to show that \(\cone(f)\) is locally contractible,
  that is, for each \(n\), there is an \(m \geq n\) such that for
  all \(i \in I_m\), the structure map
  \begin{equation}
    \label{eq:to_be_made_null-homotopic}
    \cone(f)_{m,(i,j)}^n = C[-1]_{m,i} \oplus D_{m,i}
    \xrightarrow{\gamma_{m,i}^n \oplus \eta_{m,i}^n}
    C[-1]_n \oplus D_n = \cone(f)_n
  \end{equation}
  is null-homotopic.  Here \(\gamma\) and~\(\eta\) are the structure
  maps of \(C\) and~\(D\), respectively.
  Let~\(\delta^{\cone(f)_n}\) denote the boundary map of the
  cone of~\(f_n\).  Since \(h_{m,i}^C\), \(h_{m,i}^D\) are local
  chain homotopies between \(g_{m,i} \circ f_{m,i}\) and
  \(\gamma_{m,i}^n\), and \(f_n \circ g_{m,i}\) and
  \(\eta_{m,i}^n\), respectively, the matrix
  \[
    \tilde{h}_{m,i} = \begin{pmatrix}
      -h_{m,i}^{C[-1]} & g_{m,i}^n\\
      0 & h_{m,i}^D
    \end{pmatrix}
    \colon \cone(f)_{m,(i,j)} \to \cone(f)_n
  \]
  satisfies
  \[
    \delta^{\cone(f)_n} \circ \tilde{h}_{m,i} + \tilde{h}_{m,i}
    \circ \delta^{\cone(f)_{m,(i,j)}} \\
    =  \begin{pmatrix}
      \gamma_{m,i}^n &  h_{m,i}^D \circ f_{m,i} - f_n \circ h_{m,i}^C \\
      0 & \eta_{m,i}^n
    \end{pmatrix}.
  \]
  Then we compute that \(h = \tilde{h} \circ \Psi\) with
  \[
    \Psi_{m,i}^n \defeq \begin{pmatrix}
      \eta_{m,i}^n &  f_n \circ h_{m,i}^C - h_{m,i}^D \circ f_{m,i}\\
      0 & \gamma_{m,i}^n
    \end{pmatrix},
  \]
  is the desired null-homotopy for~\eqref{eq:to_be_made_null-homotopic}.
\end{proof}

\begin{lemma}\label{lem:main_local_chain_htpy_equi}
  The chain map
  \((s^\#,\id_A)_*\colon \HAC(A) \to \HAC(A,W,\varrho)\) is a local
  chain homotopy equivalence.
\end{lemma}

\begin{proof}
  Recall the notation \(\HAC(A,W,\varrho,l)\) introduced
  in~\eqref{eq:HAC_l}.  We know that
  \begin{equation}
    \label{eq:indlim_local}
    \HAC(A,W,\varrho,l) \cong \varinjlim
    \HAC(A,M,\varrho|_M,l),\qquad
    \HAC(A) \cong \varinjlim \HAC(A,\dvr[X],\id_A^\#|_{\dvr[X]},l)
  \end{equation}
  for all fixed \(l\in\N^*\), where the inductive systems are
  indexed by the \(\dvgen\)\nb-adically complete, bounded
  \(\dvr\)\nb-submodules \(M \subseteq W\) and the finite subsets
  \(X\subseteq A\), respectively.  The chain maps \((f,g)_*\)
  produced by Lemma~\ref{lem:hybrid_functorial} are levelwise, that
  is, they are represented by a coherent family of chain maps for
  fixed \(l\in\N^*\).  The elementary homotopies of algebra
  homomorphisms in Proposition~\ref{pro:homotopy_lifting_2} are
  levelwise as well.  The induced chain homotopies are, however, not
  quite levelwise.  The algebras \(\tub{\tens W}{I^l} \otimes \dvf\)
  for \(l \in \N^*\) are not quasi-free for fixed~\(l\).  Therefore,
  we cannot directly apply the homotopy invariance of the
  \(X\)\nb-complex to the local homomorphisms for a fixed
  \(l\in\N^*\) built above.

  In Lemma~\ref{lem:hybrid_functorial} and
  Proposition~\ref{pro:homotopy_lifting_2}, we already stated some
  intermediate results.  Namely, \(s^\#\colon \dvr[A] \to W\),
  \(f_M\colon M \to \dvr[A]\) and \(s_X^\#\colon \dvr[X] \to W\) induce
  bounded \(\dvf\)\nb-algebra homomorphisms
  \begin{align*}
    (s^\#)_*\colon
    \ling{\tub{\tens\dvr[A]}{J^l}}\otimes \dvf
    &\to \ling{\tub{\tens W}{I^l}} \otimes \dvf,\\
    (f_M)_*\colon
    \ling{\tub{\tens M}{(\tens M \cap I)^l}} \otimes \dvf
    &\to \ling{\tub{\tens\dvr[A]}{J^l}}\otimes \dvf,\\
    (s_X^\#)_*\colon
    \ling{\tub{\tens\dvr[X]}{(J\cap \tens\dvr[X])^l}}\otimes \dvf
    &\to \ling{\tub{\tens M}{(\tens M \cap I)^l}} \otimes \dvf.\\
    \shortintertext{And the composite maps}
    (s^\#)_*\circ (f_M)_*\colon
    \ling{\tub{\tens M}{(\tens M \cap I)^l}} \otimes \dvf
    &\to \ling{\tub{\tens W}{I^l}} \otimes \dvf,\\
    (f_M)_* \circ (s_X^\#)_*\colon
    \ling{\tub{\tens\dvr[X]}{(J\cap \tens \dvr[X])^l}}\otimes \dvf
    &\to \ling{\tub{\tens\dvr[A]}{J^l}}\otimes \dvf
  \end{align*}
  are homotopic to the canonical inclusion maps.

  In the following argument, we
  tacitly turn bornological \(\dvf\)\nb-vector spaces into inductive
  systems of seminormed \(\dvf\)\nb-vector spaces through the
  dissection functor in~\eqref{eq:dissection_over_F} in order to
  speak about local chain homotopy equivalences. Fix \(l \geq 1\).

  Let \(\mathcal{X}^{(2)}(D)\) for a pro-algebra~\(D\) be the
  truncated \(B\)-\(b\)-bicomplex of~\(D\) as in
  \cite{Meyer:HLHA}*{Definition~A.122}.  It is linked to the
  \(X\)\nb-complex by a canonical projection
  \(p_D\colon \mathcal{X}^{(2)}(D) \to X(D)\).  This projection is a
  chain homotopy equivalence if~\(D\) is quasi-free.  The bounded
  algebra homomorphisms above induce a commuting square of chain
  maps
  \[
    \begin{tikzcd}[column sep=8em]
      \mathcal{X}^{(2)}(\ling{\tub{\tens M}{(\tens M \cap I)^l}}
      \otimes \dvf)
      \ar[r, "\mathcal{X}^{(2)}((s^\#)_*\circ (f_M)_*)"]
      \ar[d, "p_{\ling{\tub{\tens M}{(\tens M \cap I)^l}}\otimes \dvf}"] \ar[dr, "\psi_1"] &
      \mathcal{X}^{(2)}(\ling{\tub{\tens W}{I^l}} \otimes \dvf)
      \ar[d, "p_{\ling{\tub{\tens W}{I^l}}\otimes \dvf}"]
      \\
      X(\ling{\tub{\tens M}{(\tens M \cap I)^l}} \otimes \dvf)
      \ar[r, "X((s^\#)_*\circ (f_M)_*)"] &
      X(\ling{\tub{\tens W}{I^l}} \otimes \dvf).
    \end{tikzcd}
  \]
  The canonical inclusion map \(i_M \colon \ling{\tub{\tens M}{(\tens M \cap I)^l}}
      \otimes \dvf \to \ling{\tub{\tens M}{(\tens M \cap I)^l}} \otimes \dvf\) induces another such commuting square;
  let~\(\psi_2\) denote the diagonal map in that square.  Since
  \((s^\#)_*\circ (f_M)_*\) is homotopic to the canonical inclusion
  homomorphism \(i_M\), the two diagonal chain maps
  \[
    \psi_1,\psi_2\colon
    \mathcal{X}^{(2)}(\ling{\tub{\tens M}{(\tens M \cap I)^l}}
    \otimes \dvf) \rightrightarrows
    X(\ling{\tub{\tens W}{I^l}} \otimes \dvf)
  \]
  are chain homotopic.  This is shown in the proof of
  \cite{Cortinas-Meyer-Mukherjee:NAHA}*{Proposition~4.6.1}, see also
  the proof of \cite{Meyer:HLHA}*{Theorem~4.27}.  Both results cited
  above assume quasi-freeness, but use it only to prove that the
  vertical projections from \(\mathcal{X}^{(2)}\) to~\(X\) are
  homotopy equivalences.  When we vary the tube algebra
  parameter~\(l\), our algebras become quasi-free:

  \begin{lemma}
    \label{tube-tensor-quasi_free}
    Let \(J\defeq \ker (\id_A^\#\colon \tens \dvr[A] \to A)\) and
    \(I\defeq \ker (\varrho^\#\colon \tens W \to A)\).  The
    pro-algebras \(\ling{\tub{\tens\dvr[A]}{J^\infty}}\),
    \(\ling{\tub{\tens W}{I^\infty}}\),
    \(\ling{\tub{\tens\dvr[A]}{J^\infty}} \otimes \dvf\) and
    \(\ling{\tub{\tens W}{I^\infty}} \otimes \dvf\) are quasi-free.
  \end{lemma}

  \begin{proof}
    The proof is similar to the proof of
    \cite{Cortinas-Meyer-Mukherjee:NAHA}*{Proposition~4.4.6}.  The
    properties of being nilpotent mod~\(\dvgen\) and semidagger are
    both hereditary for extensions of pro-algebras, by
    \cite{Cortinas-Meyer-Mukherjee:NAHA}*{Proposition~4.2.5} and by
    adapting the proof of
    \cite{Cortinas-Meyer-Mukherjee:NAHA}*{Proposition~4.3.13}.  This
    shows that \(\ling{\tub{\tens\dvr[A]}{J^\infty}}\) and
    \(\ling{\tub{\tens W}{I^\infty}}\) are quasi-free.  Tensoring
    with~\(\dvf\) preserves quasi-freeness.
  \end{proof}

  Therefore, the canonical projection
  \[
    p_{\ling{\tub{\tens W}{I^\infty}} \otimes \dvf}\colon
    \mathcal{X}^{(2)}(\ling{\tub{\tens W}{I^\infty}} \otimes \dvf) \to
    X(\ling{\tub{\tens W}{I^\infty}} \otimes \dvf)
  \]
  is a chain homotopy equivalence.  The homotopy inverse of
  \(p_{\ling{\tub{\tens W}{I^\infty}} \otimes \dvf}\) is a morphism
  of projective systems of chain complexes.  Therefore, it consists
  of chain maps
  \(\varphi_{k_l}^l\colon X(\ling{\tub{\tens W}{I^{k_l}}} \otimes
  \dvf) \to \mathcal{X}^{(2)}(\ling{\tub{\tens W}{I^l}} \otimes
  \dvf)\) for suitable \(k_l\in\N\).  Precomposing the chain maps
  \(p_{\ling{\tub{\tens W}{I^\infty}} \otimes \dvf} \circ
  \mathcal{X}^{(2)}((s^\#)_*\circ (f_M)_*)\) and
  \(p_{\ling{\tub{\tens W}{I^\infty}} \otimes \dvf} \circ
  \mathcal{X}^{(2)}(i_M)\) with
  \(\varphi_{k_l,M}^l \colon X(\ling{\tub{\tens M}{(\tens M \cap
      I)^{k_l}}} \otimes \dvf) \to
  \mathcal{X}^{(2)}(\ling{\tub{\tens M}{(\tens M \cap I)^l}} \otimes
  \dvf)\), and using the commutativity of the diagram above, we see
  that the two chain maps
  \[
    \begin{tikzcd}[column sep=9em]
      X(\ling{\tub{\tens M}{(\tens M \cap I)^{k_l}}} \otimes \dvf)
      \ar[r, shift left, "X((s^\#)_*\circ (f_M)_*) \circ \mathsf{can}_{k_l,M}^l"]
      \ar[r, shift right, "X(i_M) \circ \mathsf{can}_{k_l,M}^l"']&
      X(\ling{\tub{\tens W}{I^l}} \otimes \dvf)      
    \end{tikzcd}
  \]
  are chain homotopic.  Here \(\mathsf{can}_{k_l,M}^l\) are the
  structure maps on the \(X\)\nb-complex.  Similarly, we get a chain
  homotopy between the two maps
  \[
    \begin{tikzcd}[column sep=7.5em]
      X(\ling{\tub{\tens\dvr[X]}{(J\cap \tens\dvr[X])^{k_l}}}\otimes \dvf)
      \ar[r, shift left, "X((f_M)_* \circ (s_X^\#)_*) \circ \mathsf{can}_{k_l,X}^l"]
      \ar[r, shift right, "X(i_X)\circ \mathsf{can}_{k_l,X}^l"']&
      X(\ling{\tub{\tens\dvr[A]}{J^l}}\otimes \dvf).
    \end{tikzcd}
  \] 
  Since \(M\) and~\(X\) are arbitrary bounded submodules, the
  criterion in Lemma~\ref{lem:explicit_local_homotopy_equivalences}
  shows that
  \((s^\#,\id_A)_*\colon X(\ling{\tub{\tens \dvr[A]}{J^\infty}}
  \otimes \dvf)\to X(\ling{\tub{\tens W}{I^\infty}} \otimes \dvf)\)
  is a local chain homotopy equivalence.

  The chain complexes \(\HAC(A)\) and \(\HAC(A,W,\varrho)\) are the
  entrywise completions of
  \(X(\ling{\tub{\tens\dvr[A]}{J^\infty}} \otimes \dvf)\) and
  \(X(\ling{\tub{\tens W}{I^\infty}} \otimes \dvf)\), respectively.
  The entrywise completion for pro-ind systems is a functor that
  preserves local chain homotopy equivalences because it is defined
  entrywise.  Therefore, the map \(\HAC(A) \to \HAC(A,W,\varrho)\)
  induced by~\(s\) is a local chain homotopy equivalence as well.
\end{proof}

Lemmas \ref{lem:locally_contractible_exact} and ~\ref{lem:main_local_chain_htpy_equi} finish the proof of
Theorem~\ref{the:HA_through_lift}.

\section{Liftings to pro-dagger algebras}
\label{sec:lift_pro_dagger}

Let~\(A\) be an \(\resf\)\nb-algebra.  If there is a projective
system of dagger algebras~\(D\) as in
Theorem~\ref{the:pro-dagger_lifting}, then the theorem gives us a
quasi-isomorphism \(\HAC(A) \simeq \diss \HAC(D)\).  We want to use
this to transfer properties of analytic cyclic homology for
projective systems of dagger algebras such as excision and matrix
stability to \(\HAC(A)\).  For this to work, we need that such a
projective system of dagger algebras exists for any~\(A\).  This is
what we are going to prove in this section.

Let~\(A\) be an \(\resf\)\nb-algebra, let~\(W\) be a torsionfree
complete bornological \(\dvr\)\nb-module, and let
\(\varrho\colon W\to A\) be a surjective bounded \(\dvr\)\nb-linear
map.  Let \(R\defeq \tens W\) and let
\(I\defeq \ker (\varrho^\#\colon R \onto A)\).  The tube pro-algebra
\(\tub{R}{I^\infty}\) contains \(\tub{I}{I^\infty}\) as an ideal.
The main issue in this section is to prove that the canonical map
from~\(A \cong R/ I\) to the quotient \(\tub{R}{I^\infty}/\tub{I}{I^\infty}\)
is an isomorphism.  This is nontrivial because the universal
property of tube algebras only gives maps to \emph{torsionfree}
pro-algebras, but not to~\(A\).  Taking linear growth bornologies
and completions, we would then get an extension of pro-dagger algebras
\[
  \comb{\ling{\tub{I}{I^\infty}}} \into
  \comb{\ling{\tub{R}{I^\infty}}} \onto
  A
\]
with analytically nilpotent kernel, as needed in
Theorem~\ref{the:pro-dagger_lifting}.

In the following, let~\(W\) be a bornological \(\dvr\)\nb-module.
Let \(W^+ \defeq W\oplus \dvr\) and let \(\Omega^0(W)\defeq W\) and
\(\Omega^l(W) \defeq W^+\otimes W^{\otimes l}\) for \(l>0\).  Let
\[
  \Omega(W) \defeq \bigoplus_{l=0}^\infty \Omega^l (W),\qquad
  \Omega^\even(W) \defeq \bigoplus_{l=0}^\infty \Omega^{2l} (W),\qquad
  \Omega^\odd(W) \defeq \bigoplus_{l=0}^\infty \Omega^{2l+1} (W).
\]
We are going to define a bornological \(\dvr\)\nb-module isomorphism
\[
\iota\colon \Omega^\even (W) \overset{\cong}\to \tens W,
\]
using a \(\dvr\)\nb-bilinear multiplication map
\(\mu \colon W \otimes W \to W\).  This is well known if~\(\mu\) is
associative (see~\cite{Cuntz-Quillen:Algebra_extensions}).
When~\(\mu\) is not associative, then the multiplication
of~\(\tens W\) transported to~\(\Omega^\even (W)\) through~\(\iota\)
becomes more complicated.  Let~\(\otimes\) denote the multiplication
in~\(\tens W\).  For \(x_1,x_2\in W\), we let
\[
  \iota(\diff x_1 \,\diff x_2) \defeq \mu(x_1,x_2) - x_1\otimes x_2
  \in \tens W.
\]
We extend this to general differential forms by
\begin{align*}
  \iota(\diff x_1\,\diff x_2 \ldots \diff x_{2l-1}\,\diff x_{2l})
  &\defeq \iota(\diff x_1\,\diff x_2) \otimes \dotsb \otimes \iota(\diff x_{2l-1}\,\diff x_{2l}),\\
  \iota(x_0 \,\diff x_1\,\diff x_2 \ldots \diff x_{2l-1}\,\diff x_{2l})
  &\defeq x_0 \otimes \iota(\diff x_1\,\diff x_2) \otimes \dotsb \otimes \iota(\diff x_{2l-1}\,\diff x_{2l})
\end{align*}
for \(x_0,\dotsc,x_{2l}\in W\).  These maps
\(\Omega^{2l}(W) \to \tens W\) combine to a bounded
\(\dvr\)\nb-module homomorphism
\(\iota\colon \Omega^\even(W) \to \tens W\).

\begin{lemma}
  \label{lem:tensor_with_Fedosov}
  The map \(\iota\colon \Omega^\even(W) \to \tens W\) is a
  bornological isomorphism.
\end{lemma}

\begin{proof}
  We define filtrations on \(\tens W\) and \(\Omega^\even W\).  Let
  \((\tens W)_j\) be the \(\dvr\)\nb-submodule generated by
  \(x_1\otimes\dotsb \otimes x_m\) with \(m\le j\) and \(x_i\in W\).
  Let \((\Omega^\even W)_j\) be the \(\dvr\)\nb-submodule generated
  by \(\Omega^{2k} W\) for all~\(k\) with \(2k<j\) and by the
  ``closed'' \(2k\)\nb-forms \(\diff x_1\dotsc \diff x_{2k}\) if
  \(2k = j\).  We prove by induction on~\(j\) that~\(\iota\)
  restricts to a bornological isomorphism from
  \((\Omega^\even W)_j\) onto \((\tens W)_j\).  The assertion is
  empty for \(j=0\).  Assuming it for \(j-1\ge0\), we prove it
  for~\(j\).  We have
  \((\tens W)_j / (\tens W)_{j-1} \cong W^{\otimes j}\).  By
  definition, \(\iota\) maps \(\diff x_1 \dotsc \diff x_{2 k}\) and
  \(x_0 \,\diff x_1 \dotsc \diff x_{2 k}\) to
  \((-1)^k x_1 \otimes \dotsb \otimes x_{2 k}\) and
  \((-1)^k x_0\otimes x_1 \otimes \dotsb \otimes x_{2 k}\) modulo
  shorter terms.  So~\(\iota\) maps \((\Omega^\even W)_j\)
  into~\((\tens W)_j\) and induces a bornological isomorphism from
  \((\Omega^\even W)_j / (\Omega^\even W)_{j-1}\) onto
  \((\tens W)_j / (\tens W)_{j-1}\).  Since the category of
  bornological \(\dvr\)\nb-modules is quasi-Abelian, the Five Lemma
  is valid.  Using the induction assumption, we conclude
  that~\(\iota\) induces a bornological isomorphism from
  \((\Omega^\even W)_j\) onto~\((\tens W)_j\).  Any bounded subset
  of \(\Omega^\even W\) or~\(\tens W\) is already bounded in
  \((\Omega^\even W)_j\) or~\((\tens W)_j\) for some~\(j\).  Now the
  assertion follows.
\end{proof}

In the following, we identify \(\tens W\) and~\(\Omega^\even W\)
using~\(\iota\).  Let~\(\odot\) be the associative multiplication
on~\(\Omega^\even W\) that corresponds to the multiplication
in~\(\tens W\).  If~\(\mu\) is associative, then~\(\odot\) is the
well known \emph{Fedosov product},
\(\omega \odot \eta = \omega \eta - \diff \omega \,\diff\eta\)
(see~\cite{Cuntz-Quillen:Algebra_extensions}); here \(\omega \eta\)
denotes the usual multiplication of differential forms, which is
dictated by the Leibniz rule.  When~\(\mu\) is non-associative, we
must correct this by terms involving the associator of~\(\mu\).
Writing~\(\mu\) multiplicatively, we compute
\begin{multline}
  \label{eq:modified_Fedosov}
  (\diff x_1\, \diff x_2)\odot x_3
  = (x_1\cdot x_2 - x_1 \odot x_2) \odot x_3
  \\
  \begin{aligned}[t]
    &= (x_1 x_2)x_3 - \diff (x_1 x_2) \,\diff x_3
    - x_1 \odot (x_2\cdot x_3) + x_1 \,\diff x_2\,\diff x_3
    \\&= (x_1 x_2)x_3 - x_1 (x_2 x_3)
    + x_1 \,\diff x_2\,\diff x_3 - \diff (x_1 x_2) \,\diff x_3
    + \diff x_1 \,\diff (x_2 x_3)
  \end{aligned}
\end{multline}
for all \(x_1,x_2,x_3\in W\).

Now we choose~\(W\) as follows.  Let \(X\subseteq A\) be a basis
of~\(A\) as an \(\resf\)\nb-vector space and let \(W\defeq \dvr[X]\).
Let \(\varrho\colon \dvr[X] \to A\) be induced by the inclusion
\(X\hookrightarrow A\).  Then~\(\varrho\) is surjective and
\(\ker \varrho = \dvgen\cdot \dvr[X] = \dvgen W\).

Lift the algebra structure on~\(A\) to a nonassociative
multiplication~\(\mu\) on~\(W\) as above and use this to identify
\(\tens W \cong \Omega^\ev W\).  The induced algebra homomorphism
\(\varrho^\#\colon \tens W\to A\) satisfies
\(\varrho^\#\circ \iota(\diff x_1 \, \diff x_2)=0\) because~\(\mu\)
lifts the multiplication in~\(A\).  Hence
\(\varrho^\#\circ \iota\colon \Omega^\even W \to A\) maps
\(\sum_{n\in\N} \omega_{2n}\) with \(\omega_{2n} \in \Omega^{2n} W\)
to \(\varrho(\omega_0)\).  Therefore,
\begin{equation}
  \label{eq:kernel_tens_to_A0}
  I\defeq \ker(\tens W\to A)
  = \dvgen W \oplus \bigoplus_{j=1}^\infty \Omega^{2 j} W.
\end{equation}
In the following, the powers~\(I^m\) of~\(I\) are defined through
the multiplication in \(\tens W\), which corresponds to the modified
Fedosov product~\(\odot\) described above.

\begin{lemma}
  \label{lem:powers_I_TB}
  If \(m\in\N^*\), then
  \[
  \tub{\tens W}{I^m} =
  \sum_{j=0}^\infty \dvgen^{-\floor{j/m}} \Omega^{2 j} W,\qquad
  \tub{I}{I^m} =
  I \oplus \sum_{j=1}^\infty \dvgen^{-\floor{j/m}} \Omega^{2 j} W,
  \]
\end{lemma}

\begin{proof}
  To simplify the formulas, we adjoin a unit element to~\(\tens W\)
  and allow \(x_0 \in W^+ \defeq W\oplus \dvr\cdot 1\) to treat
  forms with and without~\(x_0\) on an equal footing.  Let
  \(I^{(+)} \defeq I\oplus \dvgen\cdot \dvr \subseteq W^+\).  So
  \(\dvgen \in I^{(+)}\).  Let
  \(I^{(m)} \defeq \sum_{j=1}^m \dvgen^{m-j} I^j\).  Induction shows
  that
  \[
  (I^{(+)})^m
  \defeq (I\oplus \dvgen\cdot \dvr)^m
  = I^{(m)} \oplus \dvgen^m\cdot \dvr.
  \]
  Then \(\tub{\tens W}{(I^{(+)})^l} = \tub{\tens W}{I^l}\) because
  \[
  \sum_{j=1}^\infty \dvgen^{-j} I^{(m j)}
  = \sum_{j=1}^\infty \sum_{k=1}^{m j} \dvgen^{m j - k - j} I^k
  = \sum_{j=1}^\infty \dvgen^{-j} I^{m j}.
  \]
  Let \(m\ge1\) and \(j\ge0\).  Let \((m-j)_+\) be \(m-j\) if
  \(m-j\ge0\) and~\(0\) otherwise.  Let \(x_0\in W^+\),
  \(x_1,\dotsc,x_{2 j}\in W\).  Then
  \begin{equation}
    \label{eq:decompose_diff_form}
    \dvgen^{(m-j)_+}x_0 \,\diff x_1\ldots \diff x_{2 j}
    = (\dvgen^{(m-j)_+}x_0) \odot \diff x_1 \diff x_2 \odot \dotsb
    \odot \diff x_{2j -1} \diff x_{2 j}
  \end{equation}
  is a product of~\(j\) terms \(\diff x_{2i-1} \diff x_{2i} \in I\)
  and the term \(\dvgen^{(m-j)_+} x_0\), which involves
  \((m-j)_+\ge m-j\) factors \(\dvgen\in I^{(+)}\).  Thus
  \(\dvgen^{(m-j)_+} x_0 \diff x_1\ldots \diff x_{2 j} \in
  I^{(m)}\).  We prove by induction on~\(m\) that these elements
  generate~\(I^{(m)}\).  This is trivial for \(m=0\), and it
  is~\eqref{eq:kernel_tens_to_A0} for \(m=1\).

  Let \(m\ge0\) and assume the assertion is shown for~\(I^{(m)}\).
  We must prove it for
  \(I^{(m+1)}=I\odot I^{(m)} + \dvgen\cdot I^{(m)}\).  The second
  summand is easy to handle.  The first is generated by products
  \(\dvgen x_0 \odot \dvgen^{(m-j)_+} y_0 \diff y_1\ldots \diff
  y_{2j}\) and
  \(x_0\, \diff x_1\dotsb \diff x_{2 i} \odot \dvgen^{(m-j)_+} y_0
  \, \diff y_1\ldots \diff y_{2 j}\) with \(i>0\).  The first
  product is
  \begin{multline*}
    \dvgen x_0 \odot \dvgen^{(m-j)_+}y_0 \, \diff y_1\ldots \diff y_{2 j}
    \\= \dvgen^{1+(m-j)_+} (x_0 y_0) \, \diff y_1\ldots \diff y_{2 j}
    -\dvgen^{1+(m-j)_+} \,\diff x_0\, \diff y_0\, \diff y_1\ldots \diff y_{2 j},
  \end{multline*}
  where both summands have the desired form.  In the second product,
  we rewrite
  \(x_0 \,\diff x_1\dotsb \diff x_{2 i} \odot \dvgen^{(m-j)_+} y_0\)
  using~\eqref{eq:modified_Fedosov}.  All associators that appear
  here belong to~\(\dvgen W\) because \(A = W/\dvgen W\) is
  associative.  Hence we get a sum of forms that have non-zero
  degree or belong to~\(\dvgen W\).  So each summand in the total
  product has the desired form.  This completes our description
  of~\(I^{(m)}\).  The assertions on tube algebras follow
  immediately.
\end{proof}

\begin{theorem}
  \label{the:prodagger_lifting}
  Let~\(A\) be an \(\resf\)\nb-algebra.  Let~\(W\) be a torsionfree
  complete bornological \(\dvr\)\nb-module that is fine
  mod~\(\dvgen\).  Let \(\varrho\colon W\to A\) be a surjective
  \(\dvr\)\nb-linear map.  Let
  \(I \defeq \ker (\varrho^\#\colon \tens W \to A)\).  There is an
  extension of pro-algebras
  \[
    \comb{\ling{\tub{I}{I^\infty}}} \into
    \comb{\ling{\tub{\tens W}{I^\infty}}} \onto
    A,
  \]
  where \(\comb{\ling{\tub{\tens W}{I^\infty}}}\) is a pro-dagger
  algebra that is fine mod~\(\dvgen\) and
  \(\comb{\ling{\tub{I}{I^\infty}}}\) is analytically nilpotent.
\end{theorem}

\begin{proof}
  Recall from~\cite{Meyer-Mukherjee:Bornological_tf} that the dagger
  completion of a projective system~\(B\) of torsionfree
  bornological algebras is \(B^\updagger \defeq \comb{\ling{B}}\).
  Therefore, \(\comb{\ling{\tub{I}{I^\infty}}}\) and
  \(\comb{\ling{\tub{\tens W}{I^\infty}}}\) are pro-dagger algebras.
  The pro-algebra \(\tub{I}{I^\infty}\) is nilpotent mod~\(\dvgen\)
  by \cite{Cortinas-Meyer-Mukherjee:NAHA}*{Proposition~4.2.3}.  This
  is inherited by \(\comb{\ling{\tub{I}{I^\infty}}}\).  Since
  \(\comb{\tub{I}{I^\infty}}\) is pro-dagger, it is analytically
  nilpotent.

  Let \(m\in\N^*\).  By Lemma~\ref{lem:fine_quotient_tensor},
  \(\tens W\) is fine mod~\(\dvgen\).  Since \(\tub{\tens W}{I^m}\)
  carries the subspace bornology from \(\tens W\otimes\dvf\), it
  follows that \(\tub{\tens W}{I^m}/ \dvgen\cdot \tens W\) carries
  the subspace bornology from
  \(\tens W\otimes\dvf / \dvgen \tens W\), which is fine by
  Lemma~\ref{lem:fine_quotients}.  Therefore, \(\tub{\tens W}{I^m}\)
  is fine mod~\(\dvgen\).  (This also follows from
  Lemma~\ref{lem:bounded_in_tensor_tube_technical1}.)  Taking the
  linear growth bornology and completing does not change the
  quotient mod~\(\dvgen\) and its bornology.  Thus
  \(\comb{\ling{\tub{\tens W}{I^\infty}}}\) remains fine
  mod~\(\dvgen\).

  It remains to build the extension of pro-algebras above.  We first
  do this in the special case \(W_0 = \dvr[X]\) for a basis~\(X\)
  of~\(A\).  Let \(I_0 \defeq \ker (\varrho\colon \dvr[X]\to A)\).
  Lemma~\ref{lem:powers_I_TB} implies that there is
  a natural extension of \(\dvr\)\nb-algebras
  \[
    \tub{I_0}{I_0^m} \into \tub{\tens W_0}{I_0^m} \overset{q}\onto A,
  \]
  where
  \(q\bigl(\sum_{j=0}^\infty \omega_j\bigr) = \varrho(\omega_0)\) if
  \(\omega_j \in \dvgen^{-\floor{j/m}} \Omega^{2 j} W_0\) for all
  \(j\in\N\).  Since~\(A\) is an \(\resf\)\nb-algebra, it is
  semidagger.  Then
  \cite{Cortinas-Meyer-Mukherjee:NAHA}*{Lemma~2.2.6} shows that the
  linear growth bornology on \(\tub{\tens W_0}{I_0^m}\) is equal to
  the linear growth bornology relative to \(\tub{I_0}{I_0^m}\) (see
  \cite{Cortinas-Meyer-Mukherjee:NAHA}*{Definition~2.2.3} for the
  definition of linear growth bornologies relative to ideals).
  Consequently, the diagram
  \[
    \ling{\tub{I_0}{I_0^m}} \into \ling{\tub{\tens W_0}{I_0^m}}
    \overset{q}\onto A
  \]
  is an extension of bornological \(\dvr\)\nb-algebras as well.  The
  proof of \cite{Cortinas-Meyer-Mukherjee:NAHA}*{Proposition~2.3.3}
  does not use the assumption that the quotient should be
  bornologically torsionfree.  Therefore, the proposition remains
  true without this assumption.  It applies to the extension above
  because~\(A\) is complete.  It follows that
  \[
    \comb{\ling{\tub{I_0}{I_0^m}}}
    \into \comb{\ling{\tub{\tens W_0}{I_0^m}}}
    \overset{q}\onto A
  \]
  is an extension of bornological \(\dvr\)\nb-algebras.
  Letting~\(m\) vary, these extensions combine to an extension of
  pro-algebras.  This finishes the proof in the special case.

  Now let \(\varrho \colon W \onto A\) be as in the statement of the
  proposition.  Choose any map of sets \(\sigma\colon A\to \tens W\)
  that splits the surjective homomorphism
  \(\varrho^\#\colon \tens W \to A\).  Then
  \(x - \sigma(\varrho^\#(x)) \in I\) for any \(x\in \tens W\), and
  this identifies \(\tens W \cong I \times A\) as a set.  If
  \(x\in \tub{\tens W}{I^l}\) then \(x=x_0 + x_1\) with
  \(x_0 \in \tens W\), \(x_1\in \tub{I}{I^l}\), and we may replace
  \(x_0\) by \(\sigma(\varrho^\#(x_0))\) and \(x_1\) by
  \(x_1 + x_0 - \sigma(\varrho^\#(x_0))\) here.  Thus any element of
  \(\tub{\tens W}{I^l}\) is of the form \(x=\sigma(a) + x_1\) for
  some \(a\in A\), \(x_1\in \tub{I}{I^l}\).  If this decomposition
  were unique, it would follow that
  \(\tub{\tens W}{I^l}\bigm/ \tub{I}{I^l} \cong A\).  Assume that
  this fails.  Then there is \(a\in A\) with \(a\neq 0\) and
  \(\sigma(a) \in \tens W \cap \tub{I}{I^l}\).  Then
  \(\{\sigma(a)\} \in \tub{\tens M}{(\tens M\cap I)^l}\) for some
  complete, bounded submodule \(M \subseteq W\)
  by~\eqref{eq:indlim_local}.  A similar proof allows us to
  arrange
  \(\{\sigma(a)\} \in \dvgen \tens M + \tub{\tens M\cap I}{(\tens
    M\cap I)^l} = \tub{\tens M\cap I}{(\tens M\cap I)^l}\) because
  \(\sigma(a)\in \tub{I}{I^l}\).  By
  Lemma~\ref{lem:local_inverses_W_to_VA}, the identity map on~\(A\)
  lifts to a bounded \(\dvr\)-module homomorphism \(M \to \dvr[A]\).
  Similarly, we may lift \(\id_A\) to a bounded \(\dvr\)-module
  homomorphism \(f\colon M\to W_0\defeq \dvr[X]\) because~\(X\) is a
  basis of~\(A\).  By Lemma~\ref{lem:hybrid_functorial}, this
  induces a \(\dvr\)\nb-algebra homomorphism
  \(\tub{\tens M}{(\tens M\cap I)^l}\to \tub{\tens W_0}{I_0^l}\) for
  \(l\in\N^*\).  This homomorphism maps
  \(\sigma(a) \in \tub{\tens M \cap I}{(\tens M\cap I)^l}\) to
  \(\tub{I_0}{I_0^l}\), which is killed by the map
  \(\tub{\tens W_0}{I_0^l} \to A\) defined above.  The latter map
  restricted to~\(\tens W_0\) is just the canonical projection
  to~\(A\), which maps the image of \(\sigma(a) \in \tens W\) in
  \(\tens W_0\) to~\(a\).  Therefore, \(a=0\), a contradiction.
  This implies that \(\tub{\tens W}{I^m}/ \tub{I}{I^l} \cong A\).
  Now we proceed as in the case of~\(W_0\) to show that this remains
  so when we take linear growth bornologies and complete.
\end{proof}

\begin{remark}
  In particular, Theorem~\ref{the:prodagger_lifting} applies to
  \(W=\dvr[A]\) and shows that Theorem~\ref{the:pro-dagger_lifting}
  works for the pro-dagger algebra lifting
  \(\tub{\tens \dvr[A]}{I^\infty}^\updagger\) of~\(A\).  That is,
  \[
    \HAC(A) \simeq \diss \comb{X}(\tub{\tens
      \dvr[A]}{I^\infty}^\updagger).
  \]
  This gives an alternative definition of \(\HAC(A)\) that stays
  within the realm of bornological \(\dvr\)\nb-modules.  However,
  this definition would make the proofs of our main theorems much
  harder because they depend on partially defined maps, which do not
  cooperate with bornological completions.
\end{remark}

\section{Homotopy invariance, matrix stability and excision}
\label{sec:homological_properties}

In this section, we show that analytic cyclic homology for
\(\resf\)\nb-algebras and its bivariant version are homotopy
invariant, stable under tensoring with algebras of finite matrices,
and satisfy excision.  This implies that our theory is Morita
invariant for unital \(\resf\)\nb-algebras.

\begin{theorem}
  \label{the:homotopy_invariance}
  Let \(A\) be an \(\resf\)\nb-algebra.  The inclusion
  \(\iota_A \colon A \to A[t]\) induces a quasi-isomorphism
  \(\HAC(A) \cong \HAC(A[t])\).  Therefore,
  \(\HA_*(B,A) \cong \HA_*(B,A[t])\),
  \(\HA_*(A,B) \cong \HA_*(A[t],B)\) for all
  \(\resf\)\nb-algebras~\(B\) and \(\HA_*(A) \cong \HA_*(A[t])\).
\end{theorem}

\begin{proof}
  Theorem~\ref{the:prodagger_lifting} gives a pro-algebra extension
  \(N \into D \onto A\), where~\(D\) is a pro-dagger algebra and
  fine mod~\(\dvgen\) and~\(N\) is analytically nilpotent.  Then
  \[
    N \hot \dvr[t]^\updagger
    \into D \hot \dvr[t]^\updagger
    \onto A \hot \dvr[t]^\updagger
  \]
  is a pro-dagger algebra extension as well by Lemma \ref{lem:tensor-exact}.  Since
  \(\dvgen\cdot A=0\), we may identify \(A \hot \dvr[t]^\updagger\)
  with \(A[t]\), equipped with the fine bornology.  The pro-algebra
  \(D\hot \dvr[t]^\updagger\) is pro-dagger by
  \cite{Cortinas-Meyer-Mukherjee:NAHA}*{Corollary 2.1.21} (which
  mostly refers to
  \cite{Cortinas-Cuntz-Meyer-Tamme:Nonarchimedean}*{Proposition~3.1.25}).
  It is fine mod~\(\dvgen\) because
  \((D\hot \dvr[t]^\updagger)/(\dvgen) \cong (D/\dvgen D)
  \otimes_\resf \resf[t]\).  The pro-algebra
  \(N\hot \dvr[t]^\updagger\) is analytically nilpotent by
  \cite{Cortinas-Meyer-Mukherjee:NAHA}*{Proposition~4.2.6}.  Now
  Theorem~\ref{the:pro-dagger_lifting} and
  \cite{Cortinas-Meyer-Mukherjee:NAHA}*{Theorem~4.6.2} imply the
  quasi-isomorphisms
  \[
    \HAC(A) \simeq \diss \HAC(D)
    \simeq \diss \HAC(D \hot \dvr[t]^\updagger)
    \simeq \HAC(A[t]).
  \]
  This easily implies the asserted isomorphisms of homology groups.
\end{proof}

Theorem~\ref{the:homotopy_invariance} also follows immediately from
Proposition~\ref{pro:homotopy_lifting_2}.

\begin{proposition}
  \label{pro:stability}
  Let~\(A\) be an \(\resf\)\nb-algebra, \(\Lambda\) a set, and
  \(\lambda \in \Lambda\).  Let~\(M_\Lambda(A)\) be the algebra of
  finitely supported \(A\)\nb-valued matrices indexed by
  \(\Lambda \times \Lambda\).  The inclusion
  \(\iota_\lambda \colon A \to M_\Lambda(A)\),
  \(a \mapsto e_{\lambda,\lambda} \otimes a\), induces a quasi-isomorphism
  \[
    \HAC(A) \cong \HAC(M_\Lambda(A)).
  \]
  Therefore, \(\HA_*(B,A) \cong \HA_*(B,M_\Lambda(A))\),
  \(\HA_*(A,B) \cong \HA_*(M_\Lambda(A),B)\) for all
  \(\resf\)\nb-algebras~\(B\) and
  \(\HA_*(A) \cong \HA_*(M_\Lambda(A))\).
\end{proposition}

\begin{proof}
  Theorem~\ref{the:prodagger_lifting} gives a pro-algebra extension
  \(N \into D \onto A\), where~\(D\) is a pro-dagger algebra and
  fine mod~\(\dvgen\) and~\(N\) is analytically nilpotent.  Then
  \(M_\Lambda(N) \into M_\Lambda(D) \onto M_\Lambda(A)\) is a
  pro-algebra extension as well.  As in the proof of
  Theorem~\ref{the:homotopy_invariance}, we see that
  \(M_\Lambda(D)\) is again a pro-dagger algebra and fine
  mod~\(\dvgen\) and \(M_\Lambda(N)\) is still analytically
  nilpotent.  Then Theorem~\ref{the:pro-dagger_lifting} and
  \cite{Cortinas-Meyer-Mukherjee:NAHA}*{Proposition~6.2} give
  quasi-isomorphisms
  \[
    \HAC(M_\Lambda(A))
    \simeq \diss \HAC(M_\Lambda (D))
    \simeq \diss \HAC(D)
    \simeq \HAC(A).
  \]
  This easily implies the asserted isomorphisms of homology groups.
\end{proof}

\begin{theorem}
  Let \(K \overset{i}\into E \overset{p}\onto Q\) be an extension of
  \(\resf\)\nb-algebras.  Then there is a natural exact triangle
  \[
    \HAC(K) \xrightarrow{i_*} \HAC(E) \xrightarrow{p_*} \HAC(Q)
    \xrightarrow{\delta} \HAC(K)[-1]
  \]
  in the derived category.  There are natural long exact sequences
  \[
    \begin{tikzcd}[baseline=(current bounding box.west)]
      \HA_0(K) \arrow[r, "i_*"] &
      \HA_0(E) \arrow[r, "p_*"] &
      \HA_0(Q) \arrow[d, "\delta"] \\
      \HA_1(Q) \arrow[u, "\delta"] &
      \HA_1(E) \arrow[l, "p_*"] &
      \HA_1(K),  \arrow[l, "i_*"]
    \end{tikzcd}
  \]
  \[
    \begin{tikzcd}[baseline=(current bounding box.west)]
      \HA_0(B,K) \arrow[r, "i_*"] &
      \HA_0(B,E) \arrow[r, "p_*"] &
      \HA_0(B,Q) \arrow[d, "\delta"] \\
      \HA_1(B,Q) \arrow[u, "\delta"] &
      \HA_1(B,E) \arrow[l, "p_*"] &
      \HA_1(B,K),  \arrow[l, "i_*"]
    \end{tikzcd}
  \]
  \[
    \begin{tikzcd}[baseline=(current bounding box.west)]
      \HA_0(K,B) \arrow[r, <-, "i^*"] &
      \HA_0(E,B) \arrow[r, <-, "p^*"] &
      \HA_0(Q,B) \arrow[d, <-, "\delta"] \\
      \HA_1(Q,B) \arrow[u, <-, "\delta"] &
      \HA_1(E,B) \arrow[l, <-, "p^*"] &
      \HA_1(K,B).  \arrow[l, <-, "i^*"]
    \end{tikzcd}
  \]
\end{theorem}

\begin{proof}
  Theorem~\ref{the:prodagger_lifting} gives a pro-algebra extension
  \(N_Q \into R_Q \overset{p_Q}\onto Q\), where~\(R_Q\) is a
  pro-dagger algebra and fine mod~\(\dvgen\) and~\(N_Q\) is
  analytically nilpotent.  Let \(s\colon Q\to E\) be an
  \(\resf\)\nb-linear section.  When we view \(E\) and~\(Q\) as
  bornological \(\dvr\)\nb-algebras, we give them the fine
  bornology.  Therefore, \(s\) is also a bounded \(\dvr\)\nb-linear
  map, and the composite
  \(\varphi \defeq s\circ p_Q\colon R_Q \to Q \to E\) is a morphism
  of projective systems of bornological \(\dvr\)\nb-modules, briefly
  called a pro-linear map.  Let \(W_K \defeq \dvr[K]\) with the fine
  bornology, equipped with the map
  \(\varrho_K\defeq \id_K^\#\colon W_K \onto K\).  Then
  \(W_E\defeq W_K \oplus R_Q\) is a projective system of torsionfree
  complete bornological \(\dvr\)\nb-modules.  The pro-linear map
  \(\varrho_E\defeq (\varrho_K, \varphi) \colon W_E = W_K \oplus R_Q
  \onto E\) is an epimorphism because
  \(\varrho_K\colon W_K \onto K\) is surjective,
  \(p_Q\colon R_Q \onto Q\) is a cokernel, and \(E\cong K\oplus Q\).
  As in the construction of the chain complex
  \(\HAC(A,W_E,\varrho_E)\) in Section~\ref{sec:homotopy_stability},
  we form the tensor algebra \(\tens W_E\) and the ideal
  \(I_E \defeq \ker(\varrho_E^\#\colon \tens W_E \to E)\), take the
  tube algebra \(\tub{\tens W_E}{I_E^\infty}\), give it the linear
  growth bornology and complete to form the pro-dagger algebra
  \[
    R_E \defeq \tub{\tens W_E}{I_E^\infty}^\updagger.
  \]
  The canonical projection \(\mathrm{pr}_2\colon W_E \to R_Q\)
  induces a pro-algebra homomorphism
  \(\mathrm{pr}_2^\#\colon \tens W_E \to R_Q\).  The commuting
  square
  \[
    \begin{tikzcd}
      W_K \oplus R_Q \arrow[r, "\mathrm{pr}_2", twoheadrightarrow]
      \arrow[twoheadrightarrow]{d}[swap]{(\varrho_K, \varphi)} &
      R_Q \arrow[d, "p_Q", twoheadrightarrow]\\
      E \arrow[r, "p", twoheadrightarrow] & Q
    \end{tikzcd}
  \]
  implies that~\(\mathrm{pr}_2^\#\) maps
  \(I_E= \ker (\tens W_E \to E)\) to \(N_Q = \ker (R_Q \to Q)\).
  Since~\(N_Q\) is analytically nilpotent, it is nilpotent
  mod~\(\dvgen\).  By
  \cite{Cortinas-Meyer-Mukherjee:NAHA}*{Proposition~4.2.4},
  \(\mathrm{pr}_2^\#\) factors uniquely through a pro-algebra
  homomorphism \(\tub{\tens W_E}{I_E} \to R_Q\).  Since~\(R_Q\) is a
  pro-dagger algebra, this factors uniquely through a pro-algebra
  homomorphism \(\tilde{p}\colon R_E \to R_Q\).  The canonical
  inclusions
  \[
    R_Q \to W_E \to \tens W_E \to \tub{\tens W_E}{I_E^\infty} \to
    \tub{\tens W_E}{I_E^\infty}^\updagger
  \]
  compose to a pro-linear section for~\(\tilde{p}\).  Let
  \(R_K \defeq \ker \tilde{p}\).  We have built an
  extension of pro-dagger algebras \(R_K \into R_E \onto R_Q\) with
  a pro-linear section.  Therefore,
  \cite{Cortinas-Meyer-Mukherjee:NAHA}*{Theorem~5.1} gives us an
  exact triangle
  \begin{equation}
    \label{eq:triangle_in_excision}
    \HAC(R_K) \to \HAC(R_E) \to \HAC(R_Q) \to \HAC(R_K)[-1].
  \end{equation}
  It remains exact when we apply the dissection functor \(\diss\)
  in~\eqref{eq:dissection_over_F_complete}.

  Theorem~\ref{the:prodagger_lifting} implies that there are
  extensions of projective systems of bornological algebras
  \(N_E \into R_E\onto E\) and \(N_Q \into R_Q\onto Q\), where
  \(N_E\) and~\(N_Q\) are analytically nilpotent.  Since
  \(\tilde{p}\) and \(p\colon E\onto Q\) are split surjective, it
  follows that there is an extension \(N_K \into R_K \onto K\) with
  \(N_K \defeq \ker (N_E \to N_Q)\).  As an ideal in an analytically
  nilpotent pro-algebra, \(N_K\) is again analytically nilpotent.
  And~\(R_K\) is a pro-dagger algebra.  An argument as in the proof
  of Theorem~\ref{the:prodagger_lifting} shows that~\(R_E\) is fine
  mod~\(\dvgen\).  This is inherited by \(R_K\) because it is a
  direct summand of~\(R_E\) as a projective system of bornological
  \(\dvr\)\nb-modules.  Now Theorem~\ref{the:pro-dagger_lifting}
  implies that the canonical maps \(R_K \onto K\), \(R_E \onto E\),
  \(R_Q \onto Q\) induce quasi-isomorphisms
  \[
    \diss \HAC(R_K) \simeq \HAC(K),\qquad
    \diss \HAC(R_E) \simeq \HAC(E),\qquad
    \diss \HAC(R_Q) \simeq \HAC(Q).
  \]
  Then all assertions follow easily from the exact triangle
  in~\eqref{eq:triangle_in_excision}.
\end{proof}

\begin{example}
  \label{exa:Leavitt_path_2}
  We consider again the Leavitt and Cohn path algebras \(L(R,E)\)
  and \(C(R,E)\) of a directed graph~\(E\) with coefficients in a
  commutative ground ring~\(R\) as in
  Example~\ref{exa:Leavitt_path_algebras}.  The results
  in~\cite{Cortinas-Montero:K_Leavitt} allow to compute the analytic
  cyclic homology of \(L(\resf,E)\) and \(C(\resf, E)\) directly,
  without any extra work because analytic cyclic homology for
  algebras over~\(\resf\) is exact, matrix-stable, and homotopy
  invariant.  In contrast, the computation of
  \(\HAC(C(\dvr,E)^\updagger)\) and \(\HAC(L(\dvr,E)^\updagger)\)
  in~\cite{Cortinas-Meyer-Mukherjee:NAHA} is based on the ``ideas of
  the proofs'' of results in~\cite{Cortinas-Montero:K_Leavitt}.
  Some extra work is needed because dagger completions are not
  exact; in fact, some completions needed for the proof are not true
  dagger completions.  Therefore, the computation of analytic cyclic
  homology over~\(\resf\) together with the isomorphisms in
  Example~\ref{exa:Leavitt_path_2} gives a more elegant proof of
  \cite{Cortinas-Meyer-Mukherjee:NAHA}*{Theorem~8.1}.
\end{example}

A useful special case of Example~\ref{exa:Leavitt_path_2} is the
following analogue of the Bass Fundamental Theorem from algebraic
\(K\)\nb-theory:

\begin{corollary}
  \label{cor:fundamental_theorem}
  For any \(\resf\)\nb-algebra~\(A\), there is a
  quasi-isomorphism
  \[
    \HAC(A \otimes \resf[t, t^{-1}]) \cong \HAC(A) \oplus
    \HAC(A)[1].
  \]
\end{corollary}

\begin{proof}
  As in the proof of Theorem~\ref{the:homotopy_invariance}, we may
  identify the fine bornological algebra
  \(A \otimes \resf[t, t^{-1}]\) with
  \(A \hot \dvr[t,t^{-1}]^\updagger\) and there is an analytically
  nilpotent extension
  \[
    N \hot \dvr[t,t^{-1}]^\updagger \into D \hot
    \dvr[t,t^{-1}]^\updagger \onto A \hot
    \dvr[t,t^{-1}]^\updagger.
  \]
  Here~\(D\) is a pro-dagger algebra lifting of~\(A\) which is
  fine mod~\(\dvgen\) and~\(N\) is analytically nilpotent.  Then
  Theorem~\ref{the:pro-dagger_lifting} implies
  \(\diss \HAC(D) \cong \HAC(A)\) and
  \(\HAC(A \otimes \resf[t,t^{-1}]) \cong \diss \HAC(D \hot
  \dvr[t,t^{-1}]^\updagger)\).  Finally, a generalisation of
  \cite{Cortinas-Meyer-Mukherjee:NAHA}*{Corollary~8.4} to pro-dagger
  algebras instead of dagger algebras gives a quasi-isomorphism
  \( \HAC(D \hot \dvr[t,t^{-1}]^\updagger) \simeq  \HAC(D) \oplus
  \HAC(D)[1]\).  This generalisation is proven in the same way.
\end{proof}

\end{document}